\documentclass[12pt,a4paper]{amsart}

\usepackage{graphicx}

\usepackage{bm}
\usepackage{xcolor}
\usepackage{mathtools}

 \usepackage[hidelinks]{hyperref}

\usepackage{amsmath,amsthm,amsfonts,amssymb,latexsym,amscd}
\usepackage[shortlabels]{enumitem}
\usepackage{slashed}

\usepackage{palatino}
\usepackage[mathcal]{euler}

\usepackage{comment}
\usepackage{tikz-cd}
\usetikzlibrary{decorations.markings}

\usepackage{xy}
\xyoption{all}

\numberwithin{equation}{subsection}
\swapnumbers

\newtheorem{theorem}[equation]{Theorem}
\newtheorem*{theorem*}{Theorem}

\newtheorem{corollary}[equation]{Corollary}
\newtheorem{proposition}[equation]{Proposition}
\newtheorem*{proposition*}{Proposition}
\newtheorem{lemma}[equation]{Lemma}
\newtheorem*{conjecture*}{Conjecture}

\theoremstyle{definition}
\newtheorem{definition}[equation]{Definition}

\newtheorem{example}[equation]{Example}

\newtheorem{remark}[equation]{Remark}

\newcommand{\R}{\mathbb{R}}
\newcommand{\C}{\mathbb{C}} 
\newcommand{\N}{\mathbb{N}}
\newcommand{\Z}{\mathbb{Z}}

\DeclareMathOperator{\Tr}{Tr}

\DeclareMathOperator{\Ad}{Ad}

\newcommand{\lgd}{\mathfrak{lgd}}
\newcommand{\ad}{\operatorname{ad}}
\newcommand{\hotimes}{\mathbin{\hat \otimes}}
\DeclareMathOperator{\Prim}{Prim}
\DeclareMathOperator{\Indecomp}{Indecomp}

\usepackage{cancel} 

\begin{document}

\title[The Lie Algebra of Locally Generated Derivations]{Homology of the Lie Algebra of Locally Generated Derivations of a Discrete and Proper Metric Space}

\author{Nigel Higson}
\address{N. HIGSON, PENN STATE UNIVERSITY, UNIVERSITY PARK, PA, USA}

 \author{Tsuyoshi Kato}
 \address{T. KATO, KYOTO UNIVERSITY, KYOTO, JAPAN}

\begin{abstract}
We associate to each proper discrete metric space $X$  a Lie algebra that acts by locally generated derivations   on an infinite tensor product of matrix algebras indexed by the points of $X$.  We compute the  homology of this Lie  algebra with trivial scalar coefficients when $X$ is the  integer lattice in $n$-dimensional Euclidean space.
\end{abstract}

\maketitle


\section{Introduction}
Let $X$ be a discrete and proper metric space, and let $\alpha \colon X \to \N$ be any function. Denote by $A_\alpha(X)$  the infinite tensor product  algebra 
\[
A_\alpha (X) = \bigotimes _{x\in X} M_{\alpha(x)}(\C) .
\]
This is  the algebraic direct limit 
\[
A_\alpha (X) = \varinjlim _{F\subseteq X} A_\alpha (F)
\]
over the directed system of finite tensor products corresponding to finite subsets $F\subseteq X$. A derivation $\delta$ of the associative algebra $A_\alpha (X)$ is \emph{locally generated} if for some $R> 0$ there is a family of elements 
\[
H_x \in A_\alpha (B_X(x,R))\qquad (x\in X),
\]
where $B_X(x,R)$ is the ball of radius $R$ about $x\in X$, such that 
\[
\delta (T) = \sum_{x\in X} [H_x,T] \qquad \forall T \in A_\alpha (X)  
\]
(the sum  on the right  is actually finite, in the sense that all but finitely many summands are  zero).  The vector space of all locally generated generated derivations is closed under the commutator bracket and so is a Lie algebra that we shall denote by $\lgd_\alpha(X)$.  See Section~\ref{sec-lgd-x} for details.

In this paper we shall study the natural direct limit Lie algebra
\[
\lgd(X) = \varinjlim_{\alpha} \lgd_\alpha (X)
\]
(see Section~\ref{sec-lgd-x} again for details). 

Our main result
is the computation of the homology of this Lie algebra, with  coefficients in the trivial module $\C$, when $X = \Z^n$.  Our result is most easily described (and proved) using the standard coalgebra structure on homology, which is obtained from the  diagonal morphism
\[
\Delta \colon \lgd(X) \longrightarrow \lgd(X) \times \lgd(X).
\]
Associated to the coproduct there is a notion of primitive subspace in homology, and  shall prove the following result:

\begin{theorem*}[See Theorem~\ref{thm-suspension-isomorphism} below]
For all $n\ge 1$ and $p\ge 1$ there are vector space isomorphisms 
\[
\Prim \bigl (H_p(\lgd (\Z^{n-1}),\C)\bigr ) \stackrel \cong \longrightarrow \Prim \bigl (H_{p+1}(\lgd (\Z^{n}), \C)\bigr)
\]
relating  the primitive parts of Lie algebra homology for the Lie algebras  $\lgd (\Z^{n-1})$ and $\lgd (\Z^{n})$. 
\end{theorem*}

By combining this with an explicit computation in the case $n{=}0$, we obtain the following result (from which the full Lie algebra homology may be computed using the Milnor-Moore theorem \cite[Thm.~5.18]{MilnorMoore65}, as discussed in Section~\ref{sec-lie-algebra-homology} below).

\begin{theorem*}[See Theorem~\ref{thm-explicit-formula-for-prinitive-part} below]
If $n\ge 0$, then
    \[
    \Prim \bigl (H_{p} (\lgd(\Z^n), \C )\bigr)  =
    \begin{cases}
        \C & p= n{+}3, n{+}5,n{+}7\dots 
        \\
         0 & \text{otherwise.}
    \end{cases}
    \]
\end{theorem*}

Derivations of  the kind that we  consider in this paper  occur in the physics literature, related to multiparticle systems with finite-range entanglement, although in physics it is usually expedient  to work with a variation  of our $\lgd(X)$. Often this is in order to study the associated  group    of ``locally generated automorphisms,'' consisting of  the endpoints of all smooth paths $\alpha_t$ ($t\in [0,1]$) of automorphisms that begin at the identity and satisfy a relation like
\[
\frac{d\alpha_t}{dt}  \in \lgd (X)\qquad \forall t\in [0,1].
\]
For this purpose it is natural to replace $\lgd(X)$ with a Lie algebra that is complete in a suitable topology. See for instance \cite{KapustinSopenkoYang21} or \cite{Kubota2025}.  But in this paper we shall limit ourselves to purely algebraic constructions. Accordingly, the results reported in this paper will be   mathematical only, and we make no claim concerning their relevance to  physics.

The proof of our theorem  is modeled closely on the computation of $K$-theory for the $C^*$-algebras that arise in coarse geometry  \cite{Roe96CBMS}, and in fact the statement of  our theorem may be   paraphrased reasonably well by the assertion 
that the homological invariants of 
$\lgd(X)$ are the same as the standard homological invariants one encounters in coarse geometry.  This is perhaps surprising, but it is consistent with  other recent and closely related mathematical discoveries; see for instance \cite{Kubota2025}.

In a bit more detail, the proof of our  theorem   has two parts.  The first is the analysis of a natural  Hopf algebra structure that may be placed on $H_*(\lgd(X),\C)$, and also on the Hochschild-Serre spectral sequence. See Section~\ref{sec-hochschild-serre}.  The second is the proof, via an Eilenberg swindle argument, of  the vanishing of $H_*(\lgd(X),\C)$ in positive degrees when $X$ is a half-space in $\Z^n$; see Theorem~\ref{thm-eilenberg-swindle-for-half-spaces}.  Both parts make  use of the fact that the definition of $\lgd(X)$ involves a direct limit over \emph{all} dimension functions $\alpha \colon X \to \N$; moreover for the second part it is essential that we allow the function $\alpha$ to be possibly unbounded.

Apart from studying the Lie algebra $\lgd(X)$, we shall also study the Lie subalgebra $\lgd^0(X)$ that fixes a given, suitably chosen state $\varphi$ on $A(X)=\varinjlim A_\alpha (X)$. See Section~\ref{sec-subalgebra-fixing-a-state}. This modest generalization fixes a small aesthetic flaw in $\lgd(X)$, which is that when $X$ is a point, the primitive part of Lie algebra homology is zero in all even degrees, but also in one odd degree, namely degree one.   By considering the Lie subalgebra of derivations that fix $\varphi$, the missing generator is recovered.

It is a pleasure to thank Yosuke Kubota and Matthias Ludewig for stimulating conversations and helpful advice.

  \subsection*{Acknowledgments}

Nigel Higson’s research was supported by the
NSF grant DMS-1952669, and
  Tsuyoshi Kato's research was supported by JSPS KAKENHI 23K22394.

\section{The Lie algebra of locally generated derivations}
\label{sec-lgd-x}

Let  $X$  be a proper discrete metric space  (\emph{proper} means that every ball contains only finitely many points).  Our aim in this section is to carefully define a Lie algebra $\lgd(X)$ comprised of ``locally generated derivations''  of an infinite tensor product of matrix algebras whose individual matrix factors are labeled by points of $X$.   

\subsection{Infinite tensor product algebra}
 
Given $n\ge 0$, we  define  
\[
A_n = \underbrace{M_2(\C) \otimes \cdots \otimes M_2 (\C)}_{\text{$n$ times}} 
\]
(we set $A_0=\C$). If $m \le n$, then we shall regard $A_m$ as embedded in $A_n$ via  the map
\begin{equation}
    \label{eq-embedding-of-matrix-algebras}
T_1\otimes \cdots \otimes T_m\longmapsto
T_1\otimes \cdots \otimes T_m \otimes \underbrace{I \otimes \cdots \otimes I.}_{\text{$n{-}m$ times}}
\end{equation}

\begin{definition}
\label{def-embeddings-of-a-f-algebras-1}
    Let $X$ be a proper, discrete metric space, and let  
\[
\alpha \colon X \longrightarrow \N 
\]
be any function.  
If $F$ is any finite subset of $X$, then define  
\[
A_\alpha(F) = \bigotimes _{x\in F} A_{\alpha(x)}.
\]
If $F'$ and $F''$ are two finite subsets of $X$ with $F'\subseteq F''$, then we shall always regard $A_\alpha (F')$ as embedded in $A_\alpha (F'')$ via the map
\begin{equation}
\label{eq-embed-a-f-prime-into-a-f-double-prime}
\begin{gathered}
    A_\alpha (F')\longrightarrow A_\alpha (F'')
    \\
 \bigotimes_{x\in F'} S_x \longmapsto  \bigotimes_{x\in F''} T_x,
\end{gathered}
\end{equation}
where 
\begin{equation*}
T_x = \begin{cases} S_x & x \in F' \\
I_x & x \notin F' 
\end{cases}
\end{equation*}
(here $S_x\in A_{\alpha(x)}$, and $I_x$ denotes the multiplicative identity element in $A_{\alpha(x)}$).
\end{definition}

\begin{definition} 
\label{def-a-alpha-of-x}
Let $X$ be a proper, discrete metric space, and let  $\alpha \colon X \to \N $
be any function.  We define  
\[
A_\alpha(X) = \varinjlim _F A_\alpha (F),
\]
where the direct limit is over the finite subsets of $X$.
\end{definition}

Given $\alpha\colon X \to \N$ as above, if $\beta\colon X \to \N$ is a second function, then we shall write 
\begin{equation}
    \label{eq-order-on-alphas}
\beta \ge \alpha \quad \Leftrightarrow \quad 
\beta (x) \ge \alpha (x) \,\,\, \forall x \in X.
\end{equation}
If $\beta \ge \alpha$, then using the formula  \eqref{eq-embedding-of-matrix-algebras} we obtain an embedding 
\begin{equation}
\label{eq-embedding-of-a-alpha-f}
\iota_{\beta,\alpha}\colon A_{\alpha}(F)\longrightarrow A_{\beta}(F)
\end{equation}
for every    finite subset $F$ of $X$.  These embeddings are compatible with inclusions $F'\subseteq F''$ in Definition~\ref{def-embeddings-of-a-f-algebras-1}, and we obtain from them an embedding of direct limit algebras  
\begin{equation}
\label{eq-embedding-of-a-alpha-x}
\iota_{\beta,\alpha}\colon  A_\alpha (X) \longrightarrow A_\beta(X).
\end{equation}
If $\gamma \ge \beta \ge \alpha$, then $\iota_{\gamma, \alpha}= \iota_{\gamma,\beta}\circ \iota_{\beta,\alpha}$, and as a result we may make the following definition: 

\begin{definition} 
\label{def-a-of-x}
Let $X$ be a proper discrete metric space.  We define 
\[
A(X) = \varinjlim A_\alpha (X),
\]
where the direct limit is over the directed system of all maps $\alpha \colon X \to \N$, using the ordering in \eqref{eq-order-on-alphas} and the morphisms in \eqref{eq-embedding-of-a-alpha-x}.
\end{definition}

\subsection{Lie algebra of locally generated derivations}

Now we are able to define the Lie algebras that we wish to study. First, we record the following simple fact.

\begin{lemma}
\label{lem-local-factors-for-disjoint-sets-commute}
Let $X$ be a proper discrete metric space and let $\alpha \colon X \to \N$ be any function. If $F$ and $F'$ are disjoint finite subsets of $X$, then \[
[A_\alpha(F), A_\alpha(F')]=0
\]
\textup{(}the commutator is calculated within $A_\alpha (X)$\textup{)}.\qed
\end{lemma}

\begin{definition} 
\label{def-locally-generated-derivation}
Let $X$ be a proper metric space and let $\alpha \colon X \to \N$ be any function.
For $R\ge 0$ and $x\in X$, denote by $B_X(x,R)$ the open ball in $X$ of radius $R$ around $x$. Define the linear space 
\[
\lgd_{\alpha,R} (X) \subseteq \operatorname{Der} (A_\alpha(X)).
\]
to be  the set of all derivations 
\begin{equation}
\label{eq-form-of-delta-1}
\delta \colon A_\alpha (X) \longrightarrow A_\alpha (X)
\end{equation}
of the associative algebra $A_\alpha (X)$ for which there exist elements 
\[
H_x \in   A_\alpha(B_X(x,R)) \qquad (x\in X)
\]
such that 
\begin{equation}
\label{eq-form-of-delta-2}
\delta (T) = \sum _{x\in X} [H_x,T] \qquad \forall T\in A_\alpha (X).
\end{equation}
Note that it follows from the construction of $A_\alpha (X)$ in Definition~\ref{def-a-alpha-of-x} as a direct limit that  if $T \in A_\alpha (X)$, then there is some finite subset $F\subseteq X$ such that $T \in A_\alpha (F)$; as a result of this and Lemma~\ref{lem-local-factors-for-disjoint-sets-commute}, in the sum above, at most finitely many of the commutators $[H_x,T]$ are nonzero.
\end{definition}

\begin{lemma} 
If $\delta\in \lgd_{\alpha,R} (X)$ and $\varepsilon\in \lgd_{\alpha,S} (X)$, then the commutaor $[\delta,\varepsilon] = \delta\varepsilon{-}\varepsilon\delta$ is an element of $ \lgd_{\alpha,R+2S}(X)$. \qed
\end{lemma}

\begin{definition} 
Let $X$ be a proper metric space and let $\alpha \colon X \to \N$ be any function. Define 
\[
\lgd_\alpha (X) = \bigcup_{R\ge 0}   \lgd_{\alpha,R} (X) ;
\]
this is a Lie algebra of derivations of the associative algebra $A_\alpha (X)$.  
\end{definition} 

Now if $\alpha \le \beta$ in the sense of \eqref{eq-order-on-alphas}, and if $\delta\in A_\alpha (X)$, so that $\delta$ has form described in \eqref{eq-form-of-delta-1} and \eqref{eq-form-of-delta-2} above, then using the embeddings 
\[
A_\alpha (B_X(x,R))\longrightarrow A_\beta (B_X(x,R)) \qquad (R \ge 0,\,\,\, x\in X)
\]
from \eqref{eq-embedding-of-a-alpha-f}, we may also regard $\delta$ as an element of $\lgd _\beta (X)$.  This defines an embedding of Lie algebras 
\begin{equation}
    \label{eq-embedding-of-lgd-alpha-into-lgd-beta}
\iota_{\beta, \alpha}\colon \lgd_\alpha (X) \longrightarrow \lgd_\beta(X)
\end{equation}
that is compatible with the actions of the two Lie algebras on $A_\alpha(X)$ and $A_\beta(X)$, respectively, and with the functorial (in $\alpha)$ embedding of $A_\alpha(X)$ into $A_\beta (X)$ in \eqref{eq-embedding-of-a-alpha-x}.

\begin{definition}
Let $X$ be a proper metric space.  We set 
\[
\lgd(X) = \varinjlim_\alpha \lgd_\alpha(X),
\]
where the direct limit is over the directed system of all $\alpha \colon X\to \N$.  This is a Lie algebra of derivations of $A(X)$.
\end{definition}

\subsection{Lie subalgebras  and ideals}

Throughout this subsection,  $X$ will be a proper metric space,  $\alpha \colon X \to \N$ will be be any function, and 
 $Y \subseteq X$ will be any subspace.

\begin{definition}
Define 
$\lgd_\alpha  (X,Y)\subseteq \lgd_\alpha (X)$
to be space of all derivations  $\delta$ for which there exist $R>0$ and   $H_y\in A_\alpha (B_X(y,R))$, for  $y\in Y$, with
\[
\delta (T) = \sum _{y\in Y} [H_y,T] \qquad \forall T\in A_\alpha (X).
\]
To be clear, we emphasize that $B_X(y,R)$ is the ball around $y$ of radius $R$ in the ambient space $X$.
\end{definition}

\begin{lemma} The subspace 
$\lgd_\alpha  (X,Y)\subseteq \lgd _\alpha (X)$
is a Lie algebra ideal. \qed
\end{lemma}

\begin{definition} 
We define 
\[
\lgd (X,Y) = \varinjlim \lgd_\alpha (X,Y) ,
\]
where the direct limit is over the directed system of all $\alpha \colon X \to \N$ using \eqref{eq-order-on-alphas} and \eqref{eq-embedding-of-a-alpha-x}.  This is a Lie algebra ideal in $\lgd (X)$.
\end{definition}

\begin{definition}
For $R\ge 0$, we set
\[ 
\operatorname{Pen}_X(Y;R) = \{\, x\in X :  \text{$d(x,y) \le R$ for some $y\in Y$}\,\}.
\]
\end{definition}

\begin{lemma}
\label{lem-lgd-ideal-is-increasing-union}
$\lgd_\alpha (X,Y) = \bigcup_{R>0} \lgd _\alpha (\operatorname{Pen}_X(Y,R))$. \qed 
\end{lemma}

\subsection{Excision isomorphism}

Continuing for a moment  with the subspace  $Y\subseteq X$,  there is a natural inclusion 
\begin{equation}
    \label{eq-functoriality-of-lgd}
\lgd_\alpha(Y) \longrightarrow \lgd_\alpha (X)
\end{equation}
that may be characterized as follows: the image $\overline \delta \in \lgd_\alpha (X)$ of a derivation $\delta \in \lgd_\alpha (Y)$ satisfies 
\[
\overline \delta (ST) = \delta(S)  T \qquad \forall S \in A_\alpha (Y),\,\,\,\forall T\in A_\alpha (X \setminus Y).
\]
This relation may be used to \emph{define} $\overline \delta $  since the images of 
$A_\alpha (Y)$ and $A_\alpha (X \setminus Y)$ in $A_\alpha (X)$ commute with one another, and the inclusions induce an  isomorphism of associative algebras 
\[
A_\alpha (Y) \otimes A_\alpha (X \setminus Y)\stackrel \cong \longrightarrow A_\alpha (X).
\]
We shall use this functoriality in what follows. 

\begin{definition}[Compare {\textup{\cite[Sec.\,1]{HigsonRoeYu93}}}]  
    Let $W$ be  a proper discrete metric space. A  decomposition $W = X \cup Z$  is  \emph{$\omega$-excisive}  if for every $R >0$ there exists some $S> 0$ such that 
    \[
    \operatorname{Pen}_W( X; R) \cap \operatorname{Pen}_W(Z; R)\subseteq \operatorname{Pen}_W (X \cap Z; S) .
    \]
\end{definition}

The following definition is taken from \cite{HigsonRoeYu93}, where an excision isomorphism is proved that involves the coarse $C^*$-algebras, or \emph{Roe  algebras}, of proper discrete metric spaces.  We shall prove a similar isomorphism for our Lie algebras of locally generated derivations.

\begin{theorem} 
    \label{thm-excision-isomorphism}
If $W = X \cup Z$ is an $\omega$-excisive decomposition of a proper discrete metric space, and if $\alpha \colon W \to \N$ is any function, then  the Lie algebra morphism
\begin{equation*}
\lgd_\alpha  (X)\big / \lgd_\alpha (X,X {\cap} Z) \stackrel \cong \longrightarrow \lgd_\alpha(W) \big / \lgd_\alpha(W,   Z)
\end{equation*}
that is induced from the inclusion of $\lgd_\alpha(X)$ into $\lgd_\alpha (W)$ is a Lie algebra  isomorphism.
\end{theorem}

We shall prove Theorem~\ref{thm-excision-isomorphism} using the following two lemmas. 

\begin{lemma}
    If $X \subseteq W$, then there is a linear map 
    \[
    E_X \colon A_\alpha (W) \longrightarrow A_\alpha (X)
    \]
    such that 
    \[
    E_X (R\cdot S\cdot T) = R \cdot E_X(S)\cdot T \qquad \forall R,T \in A_\alpha (X)\,\,\, \forall S\in A_\alpha (W),
    \]
    and such that if $S\in A_\alpha (F)$ for some subset $F\subseteq W$, then $E_X(S)\in A_\alpha (X {\cap}F)$.
\end{lemma}

\begin{proof} If $Y$ is the complement of $X$ in $W$ then, as we noted earlier, the inclusion morphisms determine an isomorphism
\[
A_\alpha (X) \otimes A_\alpha (Y)\stackrel \cong \longrightarrow A_\alpha (W).
\]
If $\Tr_Y\colon A_\alpha (Y)\to \C$ is the normalized trace, then the formula 
\[
E_X = \operatorname{id} \otimes \Tr_Y \colon A_\alpha (X) \otimes A_\alpha (Y)\longrightarrow A_\alpha (X)\otimes \C
\]
defines a morphism with the required properties.
\end{proof}

\begin{lemma}
\label{lem-technical-result-for-excision}
Let  $\delta \in \lgd_\alpha (W)$.  If $W$ is the disjoint union of  $V$ and $Y$, and if 
\begin{enumerate}[\rm (i)]

\item $\delta [A_\alpha (V)] \subseteq  A_\alpha (V)$,  and 

\item $\delta [A_\alpha (Y)] = 0$,   

\end{enumerate}
then $\delta \in \lgd_\alpha (V)$.
\end{lemma} 

\begin{proof} 
There is some $R>0$ for which may write 
\[
\delta = \sum_{w\in W} \ad_{H_w}
\]
with $H_w \in A_\alpha (B_W(w,R))$ for all $w\in W$.   Consider now the derivation 
$\varepsilon \in A_\alpha (W)$ defined by 
\[
\varepsilon = \sum_{w\in W} \ad_{E_X(H_w)}
\]
If $B_W(w,R)\cap X = \emptyset$, then $E_X(H_w)\in A_\alpha (\emptyset)$, which is to say that $H_w$ is a multiple of the identity, and so $\ad_{H_w}$ can be dropped from the sum defining $\varepsilon$.  So  
\[
\varepsilon = \sum_{w\in \operatorname{Pen}_W(X;R)} \ad_{E_X(H_w)}
\]
If we partition $\operatorname{Pen}_W(X;R)$ into finite sets $F_x $ (indexed by $x\in X)$ such that $F_x\subseteq B_W(x,R)$ and define 
\[
L_x = \sum _{ w\in F_x} E_X(H_w) ,
\]
then we may write 
\[
\varepsilon = \sum_{x\in X} \ad_{L_x},
\]
which shows that in fact $\varepsilon$ is an element of $\lgd_\alpha (X)$. 

Now if $T\in A_\alpha (Y)$, then 
\begin{equation}
    \label{eq-e-equals-d-on-y}
\varepsilon (T) = 0 = \delta (T),
\end{equation}
where the second equality is from our hypotheses on $\delta$, and in addition if $T\in A_\alpha (X)$, then 
\begin{multline}
\label{eq-e-equals-d-on-x}
\varepsilon (T) = \sum_{w\in W} [ E_X(H_w),T] = \sum_{w\in W}E_X( [ H_w,T])
\\
=
E_X \Bigl ( \sum_{w\in W} [ H_w,T]\Bigr ) = E_X(\delta (T) = \delta (T),
\end{multline}
where the last equality is again from our hypotheses on $\delta$. Since $A_\alpha(X)$ and $A_\alpha (Y)$ together generate $A_\alpha (W)$, and since   $\delta$ and $\varepsilon $ are derivations, it follows from \eqref{eq-e-equals-d-on-y} and \eqref{eq-e-equals-d-on-x} that $\delta {=} \varepsilon$, which proves the lemma.
\end{proof}

\begin{proof}[Proof of Theorem~\ref{thm-excision-isomorphism}] 
The isomorphism in the statement of the theorem is equivalent to the two identities 
\begin{equation}
\label{eq-excision-condition-1}
    \lgd_\alpha(X) \cap \lgd_\alpha (W,Z) = \lgd_\alpha (X, X{\cap} Z)
\end{equation}
and
\begin{equation}
\label{eq-excision-condition-2}
    \lgd_\alpha(X) +  \lgd_\alpha (W,Z) = \lgd_\alpha (W).
\end{equation}
It is clear from the definitions that the left-hand side of  \eqref{eq-excision-condition-1} includes the right-hand side. We shall now prove the opposite inclusion.  Let $\delta \in \lgd_\alpha(X) \cap \lgd _\alpha(W,Z)$. Let $Y_1 = W \setminus X$.  Since $\delta \in \lgd_\alpha(X)$, 
\begin{equation}
\label{eq-y-1-condition}
\delta [ A_\alpha (X)]\subseteq A_\alpha (X)\quad \text{and}\quad \delta [ A_\alpha (Y_1)] = 0.
\end{equation}
Moreover, since $\delta \in \lgd_\alpha (W,Z)$, if $S\gg 0$ and if  $Y_2=  W\setminus \operatorname{Pen}_W (Z;S)$, then 
\begin{equation}
\label{eq-y-2-condition}
\delta [A_\alpha (\operatorname{Pen}_W (Z;S))] \subseteq A_\alpha (\operatorname{Pen}_W (Z;S))\quad \text{and} \quad  \delta [ A_\alpha (Y_2)] =0 .
\end{equation}
Now, since $\delta $ is a derivation, and since $A_\alpha(Y_1{\cup}Y_2)$ is generated by $A_\alpha(Y_1)$ and $A_\alpha (Y_2)$, it follows from \eqref{eq-y-1-condition} and \eqref{eq-y-2-condition} that 
\begin{equation}
\label{eq-vanishing-on-y-one-cup-y-two}
\delta\bigl [A_\alpha(Y_1{\cup}Y_2) \bigr] =0 .
\end{equation}
In addition,   the complement of $Y_1{\cup}Y_2$ in $W$ is  $X {\cap} \operatorname{Pen}_W (Z;S)$, and 
it follows from the identity 
\[
A_\alpha (X) \cap A_\alpha (\operatorname{Pen}_W (Z;S)) =
A_\alpha (X\cap \operatorname{Pen}_W (Z;S))
\]
together with \eqref{eq-y-1-condition} and \eqref{eq-y-2-condition} that 
\begin{equation}
\label{eq-complement-y-one-cup-y-two-is-invariant}
\delta \bigl [ A_\alpha \bigl (X{\cap} \operatorname{Pen}_W (Z;S))\bigr )\bigr ]
\subseteq A_\alpha \bigl (X{\cap} \operatorname{Pen}_W (Z;S)\bigr ).
\end{equation}
Using  Lemma~\ref{lem-technical-result-for-excision}, it now follows from  \eqref{eq-vanishing-on-y-one-cup-y-two} and \eqref{eq-complement-y-one-cup-y-two-is-invariant} that   
\[
\delta \in \lgd_\alpha (X{\cap} \operatorname{Pen}_W (Z;S)).
\]
But it follows from the $\omega$-excision property that 
\[
X {\cap} \operatorname{Pen}_W (Z;S) \subseteq \operatorname{Pen}_X (X {\cap}Z;S')
\]
for some sufficiently large $S'$, and therefore 
\[
\lgd_\alpha (X {\cap} \operatorname{Pen}_W (Z;S))
\subseteq \lgd_\alpha(\operatorname{Pen}_X (X {\cap}Z;S'))
\subseteq \lgd_\alpha (X, X{\cap}Z),
\]
and so $\delta \in \lgd_\alpha (X, X{\cap}Z)$. This proves \eqref{eq-excision-condition-1}.

As for \eqref{eq-excision-condition-2}, if $\delta\in A_\alpha (W)$ and 
\[
\delta = \sum_{w\in W} \ad_{H_w},
\]
with $H_w\in B_W(w,R)$ for all $w\in W$, and if we decompose $\delta$ as a sum  
\[
\delta = \sum_{x\in X \setminus \operatorname{Pen}_R(Z,R)} \ad_{H_x} + \sum_{w\in \operatorname{Pen}_R(Z,R)} \ad_{H_w},
\]
then the first sum is an element of $\lgd_\alpha (X)$ while the second is an element of $\lgd_\alpha (W,Z)$, as required. 
\end{proof}

\subsection{Lie subalgebra fixing a state}
\label{sec-subalgebra-fixing-a-state}
Define a linear functional $\varphi$ on $M_2 (\C)$ by 
\[
\varphi \colon \begin{bmatrix} a & b \\ c & d \end{bmatrix} \longmapsto a.
\]
From the point of view of $C^*$-algebra theory, $\varphi$ is a state (and in fact a pure state) and we shall use this language from now on.  However the  characteristic property of a state, that it is a \emph{positive} linear functional, will play no role.

From $\varphi$ we obtain linear functionals on all finite tensor products of copies of $M_2(\C)$ through the formula 
\[
T_1 \otimes \cdots \otimes T_n \longmapsto \varphi (T_1) \cdots  \varphi (T_n).
\]
These \emph{tensor product states} are compatible with identity-preserving inclusions \eqref{eq-embedding-of-matrix-algebras} of one finite tensor product of copies of $M_2(\C)$ into another, and so we obtain from them \emph{infinite tensor product states}
\[
A_\alpha(X)\longrightarrow \C
\quad \text{and}\quad 
A(X)\longrightarrow \C
\]
We shall use the same symbol, $\varphi$, for these, too.

\begin{definition} 
Let $X$ be a proper discrete metric space and let $\alpha \colon X \to \N$ be any function.  We shall write 
\[
\lgd^0_\alpha (X) = \bigl \{ \, 
\delta \in \lgd_\alpha (X) : \varphi (\delta (T))=0\,\,\,\forall T \in A_\alpha (X)\,\bigr \} .
\]
\end{definition} 

Each $\lgd^0_\alpha (X)$ is a Lie subalgebra of $\lgd _\alpha (X)$.  Moreover the inclusions 
\[
\iota_{\beta,\alpha}\colon \lgd_\alpha (X) \longrightarrow \lgd_\beta (X)
\]
in \eqref{eq-embedding-of-lgd-alpha-into-lgd-beta} carry $\lgd^0_\alpha (X)$ into $\lgd^0_\alpha (X)$, and so we may make the following definition: 

\begin{definition} 
Let $X$ be a proper discrete metric space.  We shall write 
\[
\lgd^0  (X) = \varinjlim _{\alpha} \lgd^0_\alpha (X), 
\]
where the direct limit is over the directed system of all $\alpha \colon X\to \N$.
\end{definition}

If $Y$ is a subspace of a proper discrete metric space $X$, then we may similarly define
\[
\lgd^0_\alpha  (X,Y) = \lgd^0_\alpha (X) \cap \lgd_\alpha (X,Y)
\]
and 
\[
\lgd^0  (X,Y) = \varinjlim _{\alpha} \lgd^0_\alpha (X,Y), 
\]
and the excision isomorphism in Theorem~\ref{thm-excision-isomorphism} carries over without change to this ideal: 
\begin{lemma} 
    \label{lem-excision-isomorphism-0}
If $W = X \cup Z$ is an $\omega$-excisive decomposition of a proper discrete metric space, then  the Lie algebra morphism
\begin{equation*}
\lgd^0 (X)\big / \lgd^0 (X,X {\cap} Z) \stackrel \cong \longrightarrow \lgd^0(W) \big / \lgd^0(W,   Z)
\end{equation*}
that is induced from the inclusion of $\lgd^0(X)$ into $\lgd^0 (W)$ is an    isomorphism. \qed
\end{lemma}

\section{Lie algebra homology}
\label{sec-lie-algebra-homology}

In this section we shall quickly review some basic information about Lie algebra homology, partly to establish notation, and partly to introduce some variations on standard constructions that we shall use later.

\subsection{Basic definitions}
We shall work throughout with Lie algebras $\mathfrak{g}$  over  the complex numbers.  Nearly always they will be infinite-dimensional. Denote by $\mathcal{U}(\mathfrak{g})$ the enveloping algebra of $\mathfrak{g}$.  The space  
\begin{equation}
    \label{eq-def-of-v-spaces}
V_p(\mathfrak{g}) = \mathcal{U}(\mathfrak{g})\otimes  \Lambda ^p (\mathfrak{g})
\end{equation}
(tensor product over $\C$, here and everywhere else, unless otherwise indicated) is a free left $\mathcal{U}(\mathfrak{g})$-module and the formula 
\begin{multline*}
b\colon T\otimes (X_1\wedge\cdots \wedge X_p)
\longmapsto \sum_i (-1)^{i+1} TX_i \otimes (X_1\wedge\cdots \wedge \hat X_i \wedge \cdots  \wedge X_p)
\\
+\sum _{i< j} (-1)^{i +j} T  \otimes ([X_i,X_j]\wedge X_1\wedge\cdots \wedge \hat X_i \wedge \cdots \wedge \hat X_j\wedge \cdots  \wedge X_p)
\end{multline*}
defines differentials 
\begin{equation}
    \label{eq-def-of-v-differentials}
b\colon V_p(\mathfrak{g}) \longrightarrow V_{p-1}(\mathfrak{g})
\end{equation}
giving  a resolution 
\[
\xymatrix{
0 & \C \ar[l] &  V_0(\mathfrak{g})\ar[l] &  V_1(\mathfrak{g})\ar[l]_b &  V_2(\mathfrak{g})\ar[l]_b &\cdots  \ar[l]_-b
}
\]
of the trivial $\mathfrak{g}$-module by free $\mathcal{U}(\mathfrak{g})$-modules. 
See \cite[Sec.~XIII.7]{CartanEilenberg56}.  If $M$ is a left $\mathfrak {g}$-module, then 
one may define the homology groups $H_p(\mathfrak{g}, M)$ as the homology of the complex 
\[
\xymatrix{
  [V_0(\mathfrak{g})\otimes _{\C} M] _{\mathfrak{g}}   &  [V_1(\mathfrak{g})\otimes _{\C} M] _{\mathfrak{g}} \ar[l]_{b\otimes 1} &  [V_2(\mathfrak{g})\otimes _{\C} M]_{\mathfrak{g}} \ar[l]_{b\otimes 1} & \cdots \ar[l]_-{b\otimes 1}
}
\]
of coinvariants. We are   interested in the case of the trivial module $M=\C$, in which case the above complex reduces to 
\begin{equation}
    \label{eq-complex-for-homology-with-trivial-coefficients}
\xymatrix{
  \Lambda^0(\mathfrak{g})   &  \Lambda^1(\mathfrak{g})  \ar[l]_{b'} &  \Lambda^2(\mathfrak{g})  \ar[l]_{b'} &\cdots  \ar[l]_-{b'}
}
\end{equation}
with differential 
\begin{multline}
    \label{eq-differential-for-homology-with-trivial-coefficients}
b' \colon X_1\wedge\cdots \wedge X_p
\\
\longmapsto \sum _{i< j} (-1)^{i +j}   [X_i,X_j]\wedge X_1\wedge\cdots \wedge \hat X_i \wedge \cdots \wedge \hat X_j\wedge \cdots  \wedge X_p .
\end{multline}

\subsection{Actions by derivations and  automorphisms}
\label{sec-derivations-and-automorphisms}
Let $\mathfrak{g}$ be a complex Lie algebra 
and let $\delta\colon \mathfrak{g}\to \mathfrak{g}$ be a derivation.  If we define
$\delta\colon \Lambda^p(\mathfrak{g}) \to \Lambda^p(\mathfrak{g})$ by means of the formula 
\begin{equation}
    \label{eq-action-of-differential-on-c-e-complex}
\delta \colon X_1\wedge\cdots \wedge X_p \longmapsto 
\sum _{i=1}^p  X_1\wedge\cdots\wedge \delta(Y_i)\wedge \cdots  \wedge X_p,
\end{equation}
then $\delta$ commutes with the differential $b'$ in \eqref{eq-differential-for-homology-with-trivial-coefficients}, and moreover the formula \eqref{eq-action-of-differential-on-c-e-complex} determines an action of the Lie algebra of all derivations of $\mathfrak{g}$ on the chain complex \eqref{eq-complex-for-homology-with-trivial-coefficients}, and on the each of the homology groups $H_p(\mathfrak{g},\C)$.

If  $Y\in \mathfrak{g}$,  if $\ad_Y\colon \mathfrak{g}\to \mathfrak{g}$ is the associated inner derivation, and if we define $\varepsilon_Y\colon \Lambda^p(\mathfrak{g}) \to  \Lambda^{p+1}(\mathfrak{g})$ by means of the formula
\[
\varepsilon_Y \colon  X_1\wedge\cdots \wedge X_p
\longmapsto  Y\wedge X_1\wedge\cdots \wedge X_p ,
\]
then 
\begin{equation}
\label{eq-cartan-formula}
    b' \varepsilon_Y + \varepsilon _Y b' = - \ad_Y.
\end{equation}
It follows that inner derivations act trivially on homology.

If  the derivation $\ad_Y\colon\mathfrak{g}\to \mathfrak{g}$ is \emph{locally nilpotent}, which means that for every $X\in \mathfrak{g}$ there is some $n$ such that $\ad_Y^n(X) =0$, then we can form the Lie algebra automorphism 
\[
\exp (\ad_Y) \colon \mathfrak{g}\longrightarrow \mathfrak{g}
\]
using the power series for $\exp (\ad_Y)$ (which becomes a finite polynomial when applied to any $X\in \mathfrak{g}$). 
The formula  \eqref{eq-cartan-formula} integrates to the formula
\begin{multline}
    \label{eq-chain-homotopy-for-inner-automorphisms}
b' \left ( \int_0^s \exp(t \ad_Y) \varepsilon _Y \, dt \right )  + 
 \left ( \int_0^s \exp(t \ad_Y) \varepsilon _Y \, dt\right )  b'
 \\
 = \mathrm{id} - \exp(s \ad_Y) \colon \Lambda^p(\mathfrak{g}) \to \Lambda^p(\mathfrak{g}) ,
\end{multline}
valid for all $s\ge 0$.
To be clear, the chain map $\exp(t \ad_Y)\colon \Lambda ^p\mathfrak{g} \to \Lambda^p\mathfrak{g}$ is defined by means of the formula   
\[
\exp(t \ad_Y)\colon X_0\wedge \cdots \wedge X_p 
\longmapsto \exp(t \ad_Y) (X_0)\wedge \cdots \wedge \exp(t \ad_Y) (X_p).
\]
Note that when applied to any element of $\Lambda^p \mathfrak{g}$, the integrands in \eqref{eq-chain-homotopy-for-inner-automorphisms}  are polynomial functions in $t$, and so the integrals have a simple meaning.
It follows from \eqref{eq-chain-homotopy-for-inner-automorphisms} that:

\begin{lemma} 
\label{lem-inner-automorphisms-from-locally-nilpotent-elements}
Let $\mathfrak{g}$ be a complex Lie algebra. If $Y$
is a locally nilpotent element of $\mathfrak{g}$, then the Lie algebra automorphism
\[
\exp (\ad_Y) \colon \mathfrak{g} \longrightarrow \mathfrak{g}
\]
induces the identity map on $H_*(\mathfrak{g} , \C)$. \qed
\end{lemma}

\subsection{Coalgebra structure}
We continue to work with  a complex Lie algebra $\mathfrak{g}$, which may be infinite-dimensional.
The diagonal embedding 
\begin{equation}
    \label{eq-diagonal-morphism-of-lie-algebras}
\Delta \colon \mathfrak{g} \longrightarrow \mathfrak{g}\times \mathfrak{g}
\end{equation}
is a morphism of Lie algebras, and so it induces a morphism of complexes 
\[
V_* (\mathfrak{g}) \longrightarrow V_* (\mathfrak{g}\times \mathfrak{g}) ,
\]
with $(V_* (\,\,\cdot\,), b)$ as defined in \eqref{eq-def-of-v-spaces} and \eqref{eq-def-of-v-differentials}. But there is an obvious isomorphism 
\[
V_* (\mathfrak{g})\hotimes V_* (\mathfrak{g})\stackrel \cong \longrightarrow V_* (\mathfrak{g}\times \mathfrak{g}) ,
\]
leading (see \cite[Thm.\,V.10.1]{MacLane63}) to a K\"unneth  isomorphism
\[
H_*(\mathfrak{g}, \C) \hotimes H_*(\mathfrak{g}, \C) \stackrel \cong \longrightarrow H_* (\mathfrak{g}\times \mathfrak{g}, \C) ,
\]
 and so we obtain from \eqref{eq-diagonal-morphism-of-lie-algebras} a morphism of graded vector spaces 
\begin{equation}
    \label{eq-coproduct-on-homology}
\Delta \colon H_* (\mathfrak{g}, \C)\longrightarrow H_* (\mathfrak{g}, \C) \hotimes H_* (\mathfrak{g}, \C).
\end{equation}
The projection 
\begin{equation}
    \label{eq-counit-on-homology}
 \varepsilon \colon H_* (\mathfrak{g}, \C) \longrightarrow \C
\end{equation}
onto the degree zero part of homology is a counit for this coproduct, and we obtain a coalgebra structure on $H_*(\mathfrak{g}, \C)$.

\subsection{Algebra structure}
\label{subsec-algebra-structure}
It is sometimes possible to equip  $H_* (\mathfrak{g}, \C)$ with an algebra structure to go alongside its coalgebra structure. For this,  assume that there is given a Lie algebra morphism 
\begin{equation}
\label{eq-product-morphism}
\mu \colon \mathfrak{g}\times \mathfrak{g} \longrightarrow \mathfrak{g} ,
\end{equation}
that we shall call a \emph{product morphism}, and also a family of Lie algebra morphisms
\begin{equation}
\label{eq-correcting-morphism}
\gamma \colon \mathfrak{g} \longrightarrow \mathfrak{g},
\end{equation}
that we shall call \emph{correcting morphisms}, such that: 
\begin{enumerate}[label=\textup{(\theequation)}]

\refstepcounter{equation}
\item 
\label{eq-first-product-morphism-condition}
Every $\gamma$ induces the identity map on $H_* (\mathfrak{g}, \C)$.

\refstepcounter{equation}
\item If  $\iota_a$ and $\iota_2$ are the two coordinate embeddings of $\mathfrak{g}$ into $\mathfrak{g} \times \mathfrak{g}$, then the  composite Lie algebra morphisms
\[
\xymatrix{
\mathfrak{g} \ar[r]^-{\iota_1}  & \mathfrak{g} \times \mathfrak{g} \ar[r]^-{\mu} & \mathfrak{g}
}
\quad \text{and} \quad 
\xymatrix{
\mathfrak{g} \ar[r]^-{\iota_2}  & \mathfrak{g} \times \mathfrak{g} \ar[r]^-{\mu} & \mathfrak{g}
}
\]
induce the identity map in homology.

\refstepcounter{equation}
\item There is a correcting morphism $\gamma$ such that the diagram
\begin{equation*}
\xymatrix{ 
\mathfrak{g} \times \mathfrak{g}\ar[d]_{\mathrm{flip}} \ar[r]^-{\mu} & \mathfrak{g}\ar[d]^\gamma 
\\
\mathfrak{g} \times \mathfrak{g} \ar[r]_-{\mu} & \mathfrak{g}
}
\end{equation*}
is commutative.

\refstepcounter{equation}
\item 
\label{eq-last-product-morphism-condition}
There is a correcting morphism $\gamma$ such that the diagram 
\begin{equation*}
\xymatrix{ 
\mathfrak{g} \times \mathfrak{g}\times \mathfrak{g}\ar@{=}[d]  \ar[r]^-{\mu\times \mathrm{id}} & \mathfrak{g} \times \mathfrak{g}\ar[r]^-{\mu} & \mathfrak{g}\ar[d]^\gamma 
\\
\mathfrak{g} \times \mathfrak{g}\times \mathfrak{g} \ar[r]_-{\mathrm{id}\times \mu } & \mathfrak{g} \times \mathfrak{g} \ar[r]_-{\mu} & \mathfrak{g}
}
\end{equation*}
is commutative.

\end{enumerate}

\begin{lemma}
\label{lem-algebra-structure-on-homology}
Given \ref{eq-first-product-morphism-condition}-\ref{eq-last-product-morphism-condition} above, the product morphism $\mu$ defines an associative, commutative and  unital graded algebra structure on $H_* (\mathfrak{g},\C)$ \textup{(}for which the unit morphism is the inclusion of $\C$ into $H_* (\mathfrak{g},\C)$ as the degree zero part\textup{)}.
\end{lemma}

\subsection{Hopf algebra structure}
\label{sec-hopf-algebra-structure-from-product-morphism}
We turn now to Hopf algebras. 

\begin{lemma}
\label{lem-hopf-algebra-structure-on-homology}
The coproduct morphism
\[
\Delta \colon H_* (\mathfrak{g}, \C) \longrightarrow H_* (\mathfrak{g}, \C)\hotimes H_* (\mathfrak{g}, \C)
\]
in \eqref{eq-coproduct-on-homology} is an algebra morphism for the algebra structure that is given in Lemma~\textup{\ref{lem-algebra-structure-on-homology}}.
\qed
\end{lemma}

\begin{definition}
Following   \cite[Def.~4.1]{MilnorMoore65}, we shall say that  a graded vector space $H_{*}=\oplus_{p=0}^\infty H_p$   is a \emph{connected, graded Hopf algebra} if it carries the structure of a graded algebra with unit morphism $\eta\colon \C \to H_0$, if $\eta$ is an isomorphism of vector spaces, and if $H_*$ also carries the structure of a graded coalgebra, with counit $\varepsilon\colon H_0 \to \C$ that is also an isomorphism of vector spaces, in such a way that the coproduct $\Delta $ is a morphism of algebras.
\end{definition}

\begin{example} 
Lemma~\ref{lem-hopf-algebra-structure-on-homology} gives $H_* (\mathfrak{g} , \C)$ the structure of a connected Hopf algebra.
\end{example}

In the following sections we shall use  several concepts and structural results about connected, graded, commutative and cocommutative  Hopf algebras, which we shall summarize here.

\begin{definition}
If $H_*$ is any connected, graded Hopf algebra, then for $p\ge 0$ the space of degree $p$ \emph{primitive elements} in $H_*$ is
\[
\Prim (H_p)  = \bigr \{ c\in H_p : \Delta c = 1\otimes c + c \otimes 1\,\bigr  \},
\]
(note that  $\Prim (H_0)=0$), while the space of degree $p$ \emph{indecomposable elements} is
\begin{multline*}
\Indecomp (H_p) 
\\
= H_p \big / \mathrm{span} \bigl \{ c_r \cdot c_s : c_r\in H_r,\,\, c_s\in H_s,\,\, p = r{+}s, \,\,\,\, r,s> 0 \,\bigr \}.
\end{multline*}
\end{definition}

\begin{theorem}[{\cite[Cor.~4.18]{MilnorMoore65}}] 
\label{thm-prim-iso-to-indecomp}
If $H_*$ is a connected, graded Hopf algebra with commutatative multiplication and commutative comultiplication, then the canonical morphism
\[
\Prim(H_*) \longrightarrow \Indecomp (H_*)
\]
\textup{(}the composition of the inclusion of   $\Prim(H_*)$ into  $H_*$ with the quotient map from $H_*$ to $ \Indecomp (H_*)$\textup{)} is an isomorphism.
\end{theorem}

\begin{theorem}[{\cite[Thm.~5.18]{MilnorMoore65}}]
\label{thm-milnor-moore}
If $H_*$ is a connected, graded Hopf algebra with commutative multiplication and commutative comultiplication, then the canonical morphism
\[
\operatorname{Sym}(\Prim(H_*)) \longrightarrow H_*
\]
from the symmetric algebra of $\Prim (H_*)$ into $H_*$ is an isomorphism of Hopf algebras.
\end{theorem}

\begin{remark} 
The symmetric algebra in the theorem is constructed in the graded sense (taking grading degree into account), so that for example if all the elements in the primitive part have odd grading-degree, then the symmetric algebra is in fact an exterior algebra.
\end{remark}

\subsection{Convolution of morphisms}
\label{sec-convolution-product}

Let $\mathfrak{g}$ be a complex Lie algbera. Assume that it is equipped with a product morphism
\[
\mu: 
\mathfrak{g} \times \mathfrak{g} \longrightarrow \mathfrak{g}
\]
and correcting morphisms, as in Section~\ref{subsec-algebra-structure}.

Given two Lie algebra morphisms
$\varphi,\psi\colon \mathfrak{g} \to \mathfrak{g}$,  define a third morphism, 
\begin{equation}
    \label{eq-convolution-product}
\theta = \varphi\star \psi \colon \mathfrak{g} \longrightarrow \mathfrak{g}
\end{equation}
 by means of the commutative diagram 
\begin{equation} 
\label{eq-diagram-defining-theta}
  \xymatrix{
\mathfrak{g} \times \mathfrak{g}  \ar[r]^{(\varphi, \psi)} & \mathfrak{g} \times \mathfrak{g} \ar[d]^{\mu}
\\
  \mathfrak{g}  \ar[u]^{\Delta} \ar[r]_{\theta}   & 
   \mathfrak{g}
  }
\end{equation}
where $\Delta$ is the diagonal homomorphism.  In the context of Hopf algebra theory, this is usually called the \emph{convolution product} of $\varphi$ and $\psi$.  

\begin{lemma}
\label{lem-convolution-of-morphisms}
If $\phi\star \psi $ is defined as in \eqref{eq-convolution-product} and \eqref{eq-diagram-defining-theta} above, then 
\begin{equation*}
(\phi\star \psi)_* =  \varphi_* {+} \psi_* 
\colon \Prim ( H_*(\mathfrak{g} , \C))\longrightarrow \Prim ( H_*(\mathfrak{g} , \C)).
\end{equation*}
\end{lemma}

\begin{proof}
If $c \in \Prim ( H_*(\mathfrak{g} , \C))$, then the element 
$(\phi\star \psi)_*(c)$ is the image of the  composition
\[
c \mapsto c \otimes 1 + 1 \otimes c \mapsto \varphi_*(c) \otimes 1 + 1 \otimes 
\psi_*(c)  \mapsto \varphi_*(c) \cdot 1 + 1 \cdot \psi_*(c).
\]
But  $\varphi_*(c) \cdot 1 + 1 \cdot \psi_*(c) = \varphi_* (c) + \psi_* (c)$,
as required.
\end{proof}

\subsection{Abstract Eilenberg swindle argument}
\label{sec-abstract-eilenberg-swindle}

We shall continue to assume that  $\mathfrak{g}$ is equipped with a product morphism
\[
\mu: 
\mathfrak{g} \times \mathfrak{g} \longrightarrow \mathfrak{g}
\]
and correcting morphisms, as in Section~\ref{subsec-algebra-structure}.

Occasionally it is possible to construct an infinitary version of the convolution product in Section~\ref{sec-convolution-product}, with the following consequence:

\begin{lemma}
\label{lem-abstract-swindle}
If there is a Lie algebra morphism
\[
\mathrm{id}^\infty  \colon \mathfrak{g} \longrightarrow \mathfrak{g} 
\]
such that 
\[
(\mathrm{id} \star \mathrm{id}^\infty)_* = \mathrm{id}^{\infty}_*
\colon H_*(\mathfrak{g},\C) \longrightarrow  H_*(\mathfrak{g},\C),
\]
 then 
$H_p(\mathfrak{g}, \C) = 0$ for all $p> 0$.
\end{lemma} 

\begin{remark} 
The notation $\mathrm{id}^\infty$ is meant to suggest an ``infinite convolution product''
\[
\mathrm{id}^\infty = \mathrm{id} \star \mathrm{id} \star \mathrm{id} \star \cdots .
\]
For such a product, if it existed, it would be natural to expect the formula in the statement of the lemma.
\end{remark}

\begin{proof}[Proof of Lemma~\ref{lem-abstract-swindle}]
Under the assumption that 
    \[
    (\mathrm{id}\star \mathrm{id}^\infty)_{*} = \mathrm{id}^{\infty}_* \colon H_*(\mathfrak{g}, \C) \longrightarrow H_*(\mathfrak{g}, \C) ,
    \]
it  follows from Lemma~\ref{lem-convolution-of-morphisms} that 
    \[
    \mathrm{id} + \mathrm{id}^{\infty}_* = \mathrm{id}^{\infty}_* \colon \Prim\bigl (  H_*(\mathfrak{g}, \C) \bigr ) \longrightarrow \Prim\bigl (  H_*(\mathfrak{g}, \C) \bigr ),
    \]
and therefore  the identity map is equal to the zero map on primitive elements.  In other words, the space of primitive elements in homology  is zero. It therefore follows Theorem~\ref{thm-milnor-moore} that  the homology of $\mathfrak{g}$ is zero in positive degrees.
\end{proof}

\section{The Hochschild-Serre spectral sequence}
\label{sec-hochschild-serre}

We shall quickly review the Hochschild-Serre spectral sequence in Lie algebra homology (see for instance \cite{MacLane63} for more information), and describe a Hopf algebra structure that may be placed upon it in certain situations.

\subsection{Preliminaries}
 Let $\mathfrak g $ be a complex Lie algebra and let $\mathfrak{h}$ be an ideal in $\mathfrak{g}$. In the case of the trivial $\mathfrak{g}$-module $\C$ (which is all that concerns us here),  the Hochschild-Serre spectral sequence is a first-quadrant spectral sequence 
\[
E^2_{pq} = H_p (\mathfrak{g}, H_q(\mathfrak{h},\C)) \,\,\Rightarrow \,\, H_{p{+}q} (\mathfrak{g}, \C) 
\]
that is constructed as follows. Recall the spaces $V_* (\mathfrak{g}) = \mathcal{U}(\mathfrak{g})\otimes \Lambda ^* (\mathfrak{g})$ that were introducted in Section~\ref{sec-lie-algebra-homology}, and define 
\begin{equation}
\label{eq-def-of-e-zero}
E^0_{pq} =   [ V_p(\mathfrak{g}/\mathfrak{h})\hotimes [V_q(\mathfrak{g})]_{\mathfrak{h}}  ]_{\mathfrak{g}/\mathfrak{h}}
\qquad (p\ge 0,\,\,\,q\ge 0)
\end{equation}
These spaces carry  two anti-commuting differentials 
\begin{equation}
\label{eq-def-of-e-zero-differentials}
b' = b\hotimes \mathrm{id} \colon E^0_{p,q} \to E^0_{{p-1},q} \quad \text{and} \quad b'' =  \mathrm{id}\hotimes b  \colon E^0_{p,q} \to E^0_{{p},{q-1}},
\end{equation}
and we arrive at a first-quadrant double complex
\begin{equation}
    \label{eq-first-quadrant-double-cplx}
\xymatrix{
\ar[d]_{b''} & \ar[d]_{b''} & \ar[d]_{b''} & \ar[d]_{b''} & 
\\
E^0_{0,3} \ar[d]_{b''} & E^0_{1,3}\ar[l]^{b'} \ar[d]_{b''}& E^0_{2,3}\ar[l]^{b'}\ar[d]_{b''} &
E^0_{3,3}\ar[l]^{b'}\ar[d]_{b''} &   \ar[l]^{b'}
\\
E^0_{0,2} \ar[d]_{b''} & E^0_{1,2}\ar[l]^{b'} \ar[d]_{b''}& E^0_{2,2}\ar[l]^{b'}\ar[d]_{b''} &
E_{3,2}\ar[l]^{b'}\ar[d]_{b''} &   \ar[l]^{b'}
\\
E^0_{0,1} \ar[d]_{b''} & E^0_{1,1}\ar[l]^{b'} \ar[d]_{b''}& E^0_{2,1}\ar[l]^{b'}\ar[d]_{b''} &
E^0_{3,1}\ar[l]^{b'}\ar[d]_{b''} &   \ar[l]^{b'}
\\
E^0_{0,0} & E^0_{1,0}\ar[l]^{b'} & E^0_{2,0}\ar[l]^{b'} &
E^0_{3,0}\ar[l]^{b'} &   \ar[l]^{b'}}
\end{equation}
Denote  by $H'$ the homology of the rows  (with respect to $b'$), and denote by $H''$ the homology of the columns (with respect to $b''$).  As usual, there are two spectral sequences that converge to the homology of the totalization of the double complex  \eqref{eq-first-quadrant-double-cplx}: the first  has $E^2_{pq}=H''\bigl (H' (E^0_{pq})\bigr )$, and the second  has $E^2_{pq}=H'\bigl (H'' (E^0_{pq})\bigr )$.  As for the first, we have:

\begin{lemma} 
\label{lem-first-spectral-sequence}
For $p,q\ge 0$ there are vector space isomorphisms
\[
\pushQED{\qed}
H''\bigl (H' (E^0_{pq})\bigr ) \cong 
\begin{cases} H_q (\mathfrak{g}, \C) & p =0 \\
0 & p> 0 . 
\end{cases}
\qedhere\popQED
\]
\end{lemma}

So the first spectral sequence collapses at the $E^2$-term, and as a result: 

\begin{theorem}
The homology of the totalization of the double complex \eqref{eq-first-quadrant-double-cplx} is  isomorphic to the homology of the Lie algebra $\mathfrak{g}$. \qed
\end{theorem}

\subsection{Construction of the spectral sequence}
We turn now to the second spectral sequence, for which, by definition 
\[
 E^1_{p,q}  = H'' (E_{p,q}) \quad \text{and} \quad 
  E^2_{p,q}  = H'(H'' (E_{p,q})) ,
 \]
and for which, as usual, in the $r$'th page of the spectral sequence, the differentials take the form
\begin{equation} 
\label{eq-r-th-page-differentials}
d^{(r)}\colon E^r_{p,q} \longrightarrow E^r _{p{-}r,q{+}r{-}1},
\end{equation}
which is to say that they have degree $({-}r,r{-}1)$.

\begin{lemma} 
\label{lem-projective-tensor-anything-is-projective}
Let $\mathfrak{g}$ be any   Lie algebra.  If $P$ is any projective   $\mathfrak{g}$-module, and $M$ is any   $\mathfrak{g}$-module, then the tensor product $P\otimes M$ \textup{(}over $\C$\textup{)} with the diagonal action of $\mathfrak{g}$ is also projective.  \qed
\end{lemma}

\begin{lemma} 
\label{lem-homology-and-tp-by-projective-module}
Let $\mathfrak{g}$ be any \textup{(}complex\textup{)} Lie algebra, and let $P$ be a projective   $\mathfrak{g}$-module.
Let 
\[
M_0 \longleftarrow M_1 \longleftarrow \cdots \longleftarrow M_n
\]
be any complex of   $\mathfrak{g}$-modules, with homology groups $H_0(M),\dots, H_n(M)$.  The homology groups of the complex 
\[
[ P\otimes M_0 ]_{\mathfrak{g}} \longleftarrow [ P\otimes M_1 ]_{\mathfrak{g}} \longleftarrow \cdots \longleftarrow 
[ P\otimes M_n ]_{\mathfrak{g}}
\]
are isomorphic to $
[ P{\otimes} H_0(M) ]_{\mathfrak{g}}, \dots , [ P{\otimes} H_n(M) ]_{\mathfrak{g}}$ via the map that associates to any class $[p{\otimes} x]\in [ P{\otimes} H_r(M) ]_{\mathfrak{g}}$ the homology class of   $[p{\otimes} c]_{\mathfrak{g}} \in [ P{\otimes} M_r ]_{\mathfrak{g}}$, where $c\in M_r$ is any cycle representing $x\in H_r(M)$.  
\end{lemma}

 \begin{proof} 
It follows from Lemma~\ref{lem-projective-tensor-anything-is-projective}  that the functor 
\[
M \mapsto [P\otimes M]_{\mathfrak{g}}
\]
from $\mathfrak{g}$-modules to vector spaces preserves  exact sequences. Now let  $B_p\subseteq M_p$ and $C_p\subseteq M_p$ be the submodules of boundaries and cycles, respectively for the first complex in the statement of the lemma. Apply  exactness   to the short exact sequences
\[
 0 \longrightarrow C_p \longrightarrow M_p \longrightarrow B_{p-1}\longrightarrow 0 ,
\]
\[
 0 \longrightarrow B_p \longrightarrow M_p \longrightarrow M_p/B_{p}\longrightarrow 0
\]
and 
\[
 0 \longrightarrow B_p \longrightarrow C_p \longrightarrow H_p(M)\longrightarrow 0 
\]
to obtain the result.
\end{proof} 

\begin{corollary}
\label{cor-h-2-homology-of-e-zero}
Let $p,q\ge 0$. There is a unique  vector space isomorphism
\[
  [ V_p(\mathfrak{g}/\mathfrak{h})\hotimes H_q(\mathfrak{h}) ]_{\mathfrak{g}/\mathfrak{h}}
  \stackrel \cong \longrightarrow E^1_{p,q} 
\]
that associates to each class 
\[
[v\otimes x] \in   [ V_p(\mathfrak{g}/\mathfrak{h})\hotimes H_q(\mathfrak{h}) ]_{\mathfrak{g}/\mathfrak{h}} ,
\]
with $v\in V_p(\mathfrak{g}/\mathfrak{h})$ and $x\in H_q(\mathfrak{h}) $, the homology class of the $b''$-cycle 
\[
[v \otimes c]\in E^1_{pq},
\]
where $c\in V_q(\mathfrak{g})_\mathfrak{h}$ is any cycle representing $x$.    
\end{corollary}

\begin{proof} 
This is a special case of Lemma~\ref{lem-homology-and-tp-by-projective-module}.
\end{proof} 

 We obtain from Corollary~\ref{cor-h-2-homology-of-e-zero} an isomorphism of complexes 
\begin{equation}
\label{eq-e-1-commutative-diagram}
\xymatrix{
[V_0(\mathfrak{g}/\mathfrak{h})\hotimes H_q(\mathfrak{h}) ]_{\mathfrak{g}/\mathfrak{h}} \ar[d]_{\cong} &
[V_1(\mathfrak{g}/\mathfrak{h})\hotimes H_q(\mathfrak{h}) ]_{\mathfrak{g}/\mathfrak{h}} \ar[l]_-{b'}\ar[d]_{\cong} &
 \cdots \ar[l]_-{b'}  
\\
E^1 _{0q}  & E^1 _{0q}\ar[l]^{b^1}   & \cdots \ar[l]^{b^1} ,
}
\end{equation}
from which we obtain  functorial isomorphisms
\begin{equation}
    \label{eq-e2-term-of-spectral-sequence}
 H_p\bigl (\mathfrak{g}/\mathfrak{h}, H_q(\mathfrak{h}) \bigr  ) 
 \stackrel \cong \longrightarrow E^2_{p,q} \qquad (p\ge 0,\,\, q\ge 0),
\end{equation}
in which  $ H_q(\mathfrak{h})$ is regarded as a $\mathfrak{g}/\mathfrak{h}$-module  using the action described in Section~\textup{\ref{sec-derivations-and-automorphisms}}.

\begin{lemma}
\label{lem-homology-with-coeffs-in-trivial-module}
If $\mathfrak{g} / \mathfrak{h}$ acts trivially on the homology groups $H_q(\mathfrak{h}, \C)$, then there is a natural isomorphism
\[
\pushQED{\qed}
 H_p  (\mathfrak{g}/\mathfrak{h})\hotimes  H_q(\mathfrak{h})  
 \stackrel \cong \longrightarrow H_p\bigl (\mathfrak{g}/\mathfrak{h}, H_q(\mathfrak{h}) \bigr  ).
 \qedhere\popQED
\]
\end{lemma}

As a result, if $\mathfrak{g} / \mathfrak{h}$ acts trivially on the homology groups $H_q(\mathfrak{h}, \C)$, then there are   functorial isomorphisms
\begin{equation}
    \label{eq-e2-term-of-spectral-sequence-for-trivial-action-on-homology}
 H_p  (\mathfrak{g}/\mathfrak{h})\hotimes  H_q(\mathfrak{h})
 \stackrel \cong \longrightarrow E^2_{p,q} \qquad (p\ge 0,\,\, q\ge 0).
\end{equation}

\subsection{Hopf algebra structure on the Hochschild-Serre spectral sequence}
\label{sec-hopf-algebras-and-hochschild-serre-spectral-seq}

We  shall now revisit the assumptions that we made in Section~\ref{subsec-algebra-structure} in order to equip  $H_* (\mathfrak{g}, \C)$ a graded   algebra structure, and indeed the  structure of a connected Hopf algebra. We shall observe that a  strengthening of those   assumptions may be used to  equip all of the pages of the Hochschild-Serre spectral sequence with Hopf-algebra  structures.

To be clear, when  equipping a page $E^r _{*,*}$ with a coalgebra,  algebra, or Hopf algebra structure, we shall require that
\begin{enumerate}[\rm (i)]

\item all structure maps are bigrading-preserving, and 

\item the totalized spaces 
\[
H^r_s =  \bigoplus_{p+q=s} E^r _{p,q} \qquad (s=0,1,2,\dots)
\]
acquire from these the structure of a  graded coalgebra, algebra or Hopf algebra.

\end{enumerate}

 Since the coalgebra structure on homology is derived from the diagonal Lie algebra morphism $\mathfrak{g} \to \mathfrak{g} {\times}\mathfrak{g}$, and since the Hochschild-Serre spectral sequence is functorial, the following result is straightforward.

\begin{theorem}
\label{thm-coalgebra-structure-on-spectral-sequence}
Let $\mathfrak{g} $ be a complex Lie algebra and let $\mathfrak{h}$ be an ideal in $\mathfrak{g}$.
       There are cocommutative coalgebra structures on  the graded spaces $E^r_{*,*}$ for $r\ge 0$ such that: 
     \begin{enumerate}[\rm (i)]
         \item  The coalgebra structure on the graded space $ E^2_{*,* }$ corresponds to the   tensor product of the coalgebra structures on  $H_*(\mathfrak{g}/\mathfrak{h}, \C) $ and $ H_*(\mathfrak{h}, \C)$ under the isomorphism
         \[
         E^2_{*,* } =    H_*(\mathfrak{g}/\mathfrak{h}, \C)\hotimes H_*(\mathfrak{h}, \C) 
         \]
         from  \eqref{eq-e2-term-of-spectral-sequence-for-trivial-action-on-homology}.

        \item The differential $d^{(r)}$ is compatible with the comultiplication   on  $E^r_{*,*}$ in the sense that 
        \[
        \Delta d^{(r)}  c = d^{(r)}c_{(1)}\otimes c_{(2)} \pm   c_{(1)}\otimes d^{(r)}c_{(2)}\qquad (\text{Sweedler notation})
        \]
        for all \textup{(}homogeneous\textup{)} $c\in E^r_{*,*}$.

        \item The induced coalgebra structures on the edge groups 
         $E^r _{*,0}$ are compatible with one another in the sense that the diagrams 
         \[
          \xymatrix{
    E^{r+1}_{*, 0} \ar[r]^-\Delta \ar[d]_{\mathrm{incl}}& E^{r+1}_{*,0}\hotimes E^{r+1}_{*,0}  \ar[d]^{\mathrm{incl}}
    \\
    E^{r}_{*, 0} \ar[r]_-\Delta & E^{r}_{*,0}\hotimes E^{r}_{*,0}  
    }
         \]
         are commutative. \qed
\end{enumerate}
\end{theorem}

We turn now to the issue of putting an algebra structure on the spectral sequence.
We shall continue to work with a complex Lie algebra $\mathfrak{g}$ and an ideal $\mathfrak{h} \triangleleft \mathfrak{g}$.
We shall assume we are given a product morphism $\mu\colon \mathfrak{g} \times \mathfrak{g} \to \mathfrak{g}$ and correcting morphisms $\gamma \colon \mathfrak{g}\to \mathfrak{g}$, as in \eqref{eq-product-morphism} and \eqref{eq-correcting-morphism}, satisfying all of the conditions \ref{eq-first-product-morphism-condition}-\ref{eq-last-product-morphism-condition}, but  we shall also assume   that all $\mu$ and $\gamma$ restrict to the ideal $\mathfrak{h}$:
\begin{equation}
    \label{eq-restriction-condition}
\xymatrix{
\mathfrak{g}\times \mathfrak{g} \ar[r]^-\mu &  \mathfrak{g} \\
\mathfrak{h}\times \mathfrak{h} \ar[r]_-\mu \ar[u] &  \mathfrak{h}\ar[u]
}
\quad \text{and} \quad 
\xymatrix{
 \mathfrak{g} \ar[r]^-\gamma &  \mathfrak{g} \\
 \mathfrak{h} \ar[r]_-\gamma \ar[u] &  \mathfrak{h} .\ar[u]
}
\end{equation}
Furthermore, we shall suppose that all the conditions \ref{eq-first-product-morphism-condition}-\ref{eq-last-product-morphism-condition} hold for these restricted morphisms.  Finally, the restriction condition implies that all $\mu$ and $\gamma$ pass to Lie algebra morphisms for the quotient algebra $\mathfrak{g} / \mathfrak{h}$, 
\begin{equation}
    \label{eq-morphisms-on-quotients}
\xymatrix{
\mathfrak{g}/ \mathfrak{h}\times \mathfrak{g}/ \mathfrak{h} \ar[r]^-\mu &  \mathfrak{g} / \mathfrak{h} 
}
\quad \text{and} \quad 
\xymatrix{
 \mathfrak{g} / \mathfrak{h}\ar[r]^-\gamma &  \mathfrak{g} / \mathfrak{h} ,
}
\end{equation}
and we shall assume that \ref{eq-first-product-morphism-condition}-\ref{eq-last-product-morphism-condition}  hold for these induced Lie algebra morphisms, too.

\begin{definition} 
\label{def-admissible-pair}
We shall call any pair $(\mathfrak{g},\mathfrak{h})$ consisting of a Lie algebra and an ideal an \emph{admissible pair} if it is equipped with product morphisms and correcting morphisms satisfying all of the requirements that we have just listed.
\end{definition}

\begin{theorem}
\label{thm-algebra-structure-on-spectral-sequence}
       There are algebra structures on the graded spaces $E^r_{*,*}$ for $r\ge 2$ such that: 
     \begin{enumerate}[\rm (i)]
         \item The algebra structure on the graded space $ E^2_{*,* }$ corresponds to the   tensor product of the algebra structures on  $H_*(\mathfrak{g}/\mathfrak{h}, \C) $ and $ H_*(\mathfrak{h}, \C)$ under the isomorphism
         \[
         E^2_{*,* } =    H_*(\mathfrak{g}/\mathfrak{h}, \C)\hotimes H_*(\mathfrak{h}, \C) 
         \]
         from  \eqref{eq-e2-term-of-spectral-sequence-for-trivial-action-on-homology}.

        \item The differential $d^{(r)}$ is compatible with the multiplication on  $E^r_{*,*}$ in the sense that 
         \[
        d^{(r)} \mu (a\otimes b) = \mu (d^{(r)}a\otimes b) \pm \mu (a\otimes d^{(r)}b)
        \]
        for all \textup{(}homogeneous\textup{)} $a,b\in E^r_{*,*}$.

         \item The induced algebra structures on the edge spaces 
         $E^r _{0,*}$ are compatible with one another, as $r$ increases, in the sense that the diagrams 
         \[
    \xymatrix{
    E^{r}_{0,*}\hotimes E^{r}_{0,*} \ar[r]^-{\mu} \ar[d]_{\mathrm{quot}}& E^{r}_{0,*}  \ar[d]^{\mathrm{quot}}
    \\
    E^{r+1}_{0,* } \hotimes E^{r+1}_{0,*} \ar[r]_-{\mu} & E^{r+1}_{0,*},
    }
         \]
         are commutative.
     \end{enumerate}
\end{theorem}

\begin{proof}
    We need to restrict to $r\ge 2$ because the correcting morphisms on $\mathfrak{g}$, $\mathfrak{h}$ and $\mathfrak{g}/\mathfrak{h}$ that appear in Definition~\ref{def-admissible-pair} do not necessarily act as the identity on the $r{=}0$ or $r{=}1$ pages.  But by definition, and by \eqref{eq-e2-term-of-spectral-sequence-for-trivial-action-on-homology}, they act as the identity on all of the $E^2_{pq}$, and by induction on all higher $E^r_{pq}$. The theorem now follows from  the functoriality of the spectral sequence.
\end{proof}

 Combining  Theorems~\ref{thm-coalgebra-structure-on-spectral-sequence} and \ref{thm-algebra-structure-on-spectral-sequence} we have: 

\begin{theorem}
\label{thm-hopf-structure-on-hochschld-serre-sequence}
Assume that $(\mathfrak{g},\mathfrak{h})$ is admissible pair in the sense of Definition~\textup{\ref{def-admissible-pair}}. 
The coalgebra and algebra structures on $E^r_{*,*}$ for $r\ge 2$ give, together, commutative and cocommutative connected graded  Hopf algebra structures on $E^r_{*,*}$.
\end{theorem}

\subsection{Primitive element theorem}
\label{sec-primitive-element-thm}

Let $(\mathfrak{g},\mathfrak{h})$ be an admissible pair of Lie algebras, in the sense of Section~\ref{sec-hopf-algebras-and-hochschild-serre-spectral-seq}.  The purpose of this section is to prove that if $H_* (\mathfrak{g},\C)$ is trivial, then the space of primitive elements in $H_*(\mathfrak{g}/\mathfrak{h},\C)$ may be identified with the space of primitive elements in  $H_*( \mathfrak{h},\C)$, after  a degree shift.

\begin{theorem}
Let $(\mathfrak{g}, \mathfrak{h})$ be an admissible pair in the sense of Definition~\textup{\ref{def-admissible-pair}}. For all $r\ge 2$, all $p> 0$ and all $q>  0$,    $\Prim(E^r_{p,q})=0$. 
\end{theorem}

\begin{proof} 
The proof is by induction on $r$. The base case $r{=}2$ is taken care of by  Theorems~\ref{thm-prim-iso-to-indecomp} and    \ref{thm-hopf-structure-on-hochschld-serre-sequence}, since for $p>0$ and $q>0$ every element in $E^2_{pq}\cong H_p(\mathfrak{g}/\mathfrak{h}, \C)\otimes H_q(\mathfrak{h}, \C)$ vanishes in $\Indecomp (E^2_{pq})$.  

Assume  that the  theorem holds for a given $r\ge 2$.
It follows from the Milnor-Moore theorem, Theorem~\ref{thm-milnor-moore}, that $E^r_{*,*}$ is freely generated as a graded algebra by its space of primitive elements.  Since these primitive elements all lie in the edge spaces $E^r_{*,0}$ and $E^r_{0,*}$, it follows the the mutliplication map 
\begin{equation}
    \label{eq-milnor-moore-consequence-1}
E^r_{p,0} \otimes E^r_{0,q} \longrightarrow E^r_{p,q}
\end{equation}
is a vector space isomorphism for all $p,q>0$. Now suppose $c\in E^r_{p,q}$ and $d^{(r)}c=0$, so that $c$ determines a class in $E^{r+1}_{p,q}$.  Since \eqref{eq-milnor-moore-consequence-1} is in particular survective, we may write 
\begin{equation}
    \label{eq-general-d-r-closed-element}
c = \sum_{i=1}^n a_ib_i
\end{equation}
with $a_i\in E^r_{p,0}$, $b_i \in E^r_{0,q}$ for $i=1,\dots, n$, and with $\{\,b_1,\dots, b_n\,\}$ a linearly independent set.  Applying $d^{(r)}$ we find that 
\[
0 = dc = \sum_{i=1}^n (d^{(r)}a_i)\cdot b_i ,
\]
since the classes $b_1,\dots, b_n$ are $d^{(r)}$-closed by virtue of their location on the left-edge of the spectral sequence.  We claim that the multiplication map 
\begin{equation}
    \label{eq-milnor-moore-consequence-2}
E^r_{p-r,0} \otimes E^r_{0,q+r-1} \longrightarrow E^r_{p-r,q+r-1}
\end{equation}
is also an isomorphism. When $p{-}r>0$ this follows from the induction hypothesis and the Milnor-Moore theorem, as did \eqref{eq-milnor-moore-consequence-1} above. When $p{-}r =0$,  this is because $E^r_{p-r,0}=E^r_{0,0}$ is spanned by the multiplicative unit; and when $p{-}r<0$, both sides in \eqref{eq-milnor-moore-consequence-2} are zero.
Since \eqref{eq-milnor-moore-consequence-2} is in particular injective, we have 
\[
\sum_{i=1}^n (d^{(r)}a_i)\cdot b_i = 0 \quad \Rightarrow \quad d^{(r)}a_i = 0 \quad \forall i=1,\dots , n.
\]
So all of the classes $a_1,\dots, a_n$ as well as all of the classes $b_1,\dots, b_n$, are $d^{(r)}$-closed.  So the product formula \eqref{eq-general-d-r-closed-element} shows that $c$ determines the class $0$ in $\Indecomp(E^r_{p,q})$.  Hence $\Indecomp(E^r_{p,q})=0$, and therefore $\Prim(E^r_{p,q})=0$
by Theorem~\ref{thm-prim-iso-to-indecomp}.  This finishes the inductive step.
\end{proof}
 
The following  computations, which  concern the edge terms in the Hochs\-child-Serre spectral sequence, are the main steps in the argument.

\begin{lemma} 
\label{lem-bottom-edge-terms-in-spectral-sequence}
Let $(\mathfrak{g},\mathfrak{h})$ be an admissible pair of Lie algebras, as in Definition~\textup{\ref{def-admissible-pair}}. If $\mathfrak{g}/\mathfrak{h}$ acts trivially on   $H_*(\mathfrak{h},\C)$,  and  if 
$H_p (\mathfrak{g},\C) =0$ for all $p>0$, then 
for every $p\ge 2$ the inclusion map 
\[
E^p_{p,0} \longrightarrow E^2 _{p,0}
\]
is an isomorphism onto the space of primitive elements in $E^2_{p,0} = H_p(\mathfrak{g}/\mathfrak{h}, \C)$.
\end{lemma}

\begin{lemma} 
\label{lem-side-edge-terms-in-spectral-sequence}
Let $(\mathfrak{g},\mathfrak{h})$ be an admissible pair of Lie algebras, as in Definition~\textup{\ref{def-admissible-pair}}. If $\mathfrak{g}/\mathfrak{h}$ acts trivially on   $H_*(\mathfrak{h},\C)$,  and  if 
$H_p (\mathfrak{g},\C) =0$ for all $p>0$, then 
for every $p\ge 2$ the kernel of the quotient map 
\[
E^2_{0,p-1} \longrightarrow E^p _{0,p-1}
\]
is, under the identification $E^2_{0,p-1}\cong H_{p-1}(\mathfrak{h},\C)$, precisely the kernel of the quotient map 
\[
H_{p-1}(\mathfrak{h},\C)\longrightarrow   \Indecomp(H_{p-1}(\mathfrak{h},\C)).
\]
\end{lemma}

\begin{remark} 
The   proof of Lemma~\ref{lem-bottom-edge-terms-in-spectral-sequence} will require only the coalgebra structure on homology, which is available for any pair $(\mathfrak{g},\mathfrak{h})$ consisting of a Lie algebra and an ideal.  The proof of Lemma~\ref{lem-side-edge-terms-in-spectral-sequence} will require only the algebra structure, but for this we shall the extra hypothesis that $(\mathfrak{g},\mathfrak{h})$ is an admissible pair, as in Definition~\ref{def-admissible-pair}.
\end{remark}

\begin{proof}[Proof of Lemma~\ref{lem-bottom-edge-terms-in-spectral-sequence}] 
Fix $p\ge 2$. The main step towards proving the lemma is to show  that 
\begin{equation}
\label{eq-primitive-elements-are-closed}
    \Prim (E^r_{p,0})\subseteq E^{r+1}_{p,0} \qquad \forall r =2,\ldots, p-1.
\end{equation}
With this, and  in view of the commutative diagram
\begin{equation*}
    \xymatrix{
    E^{r+1}_{*, 0} \ar[r]^-\Delta \ar[d]_{\text{incl}}& E^{r+1}_{*,0}\hotimes E^{r+1}_{*,0}  \ar[d]^{\text{incl}}
    \\
    E^{r}_{*, 0} \ar[r]_-\Delta & E^{r}_{*,0}\hotimes E^{r}_{*,0} ,
    }
\end{equation*}
we find that 
\begin{equation*}
    \Prim (E^r_{p,0}) =  \Prim ( E^{r+1}_{p,0}) \qquad \forall r =2,\ldots, p-1,
\end{equation*}
and therefore that 
\begin{equation*}
    \Prim (E^2_{p,0}) = \Prim ( E^{p}_{p,0}).
\end{equation*}
But every element of  $E^{p}_{p,0}$ is necessarily primitive in the coalgebra $E^p_{* , 0}$. Indeed
\[
E^p_{1,0} = E^\infty_{1,0} , \quad 
E^p_{2,0} = E^\infty_{2,0} , \quad \ldots   \qquad    E^p_{p-1,0} =E^\infty_{p-1,0} , 
\]
and therefore 
\[
E^p_{1,0} =  
E^p_{2,0} =   \cdots   = E^p_{p-1,0} =0 , 
\]
thanks to the assumption that $H_p(\mathfrak{g},\C)$ for all $p>0$.  It is therefore a consequence of \eqref{eq-primitive-elements-are-closed} that 
\[
 \Prim (E^2_{p,0}) =   E^{p}_{p,0},
\]
as required.

To prove \eqref{eq-primitive-elements-are-closed}, let $c\in \Prim (E^r_{p,0})$  for some $r\in \{\, 2,\dots, p{-}1\,\}$ and compute 
\[
\Delta d^{(r)} c = d^{(r)}\Delta c = d^{(r)}c \otimes 1 + 1 \otimes d^{(r)} c ,
\]
which shows that  
$d^{(r)}c$ is primitive.  However $d^{(r)}c\in E^r_{p-r,r-1}$, and since $p{-}r>0$ and $r{-}1> 0$, no nonzero element of $E^r_{p-r,r-1}$ is primitive.  Hence $d^{(r)}c=0$, and so $c\in E^{r+1}_{p,0}$, as required.
\end{proof}

\begin{proof}[Proof of Lemma~\ref{lem-side-edge-terms-in-spectral-sequence}]
The proof  is similar to the proof of Lemma~\ref{lem-bottom-edge-terms-in-spectral-sequence}. Let $p\ge 2$ and let $r\in \{\, 2,\ldots, p-1\,\}$.   The commutative diagram 
\begin{equation*}
    \xymatrix{
    E^{r}_{0,*}\hotimes E^{r}_{0,*} \ar[r]^-{\mu} \ar[d]_{\text{quot}}& E^{r}_{0,*}  \ar[d]^{\text{quot}}
    \\
    E^{r+1}_{0,* } \hotimes E^{r+1}_{0,*} \ar[r]_-{\mu} & E^{r+1}_{0,*},
    }
\end{equation*}
makes it clear that the quotient map from $E^{r}_{0,p-1}$ to $E^{r+1}_{0,p-1}$ induces a surjective map
\[
\Indecomp (E^{r}_{0,p-1}) \longrightarrow \Indecomp(E^{r+1}_{0,p-1}).
\]
We claim first that this surjective map is also injective.  In other words, we claim that every element of the image of $d^{(r)}$ in  $E^r_{0,p-1}$  is a sum of products of elements with degrees lower than $p{-}1$, and so determines  the zero element in $\Indecomp (E^{r}_{0,p-1})$.  Suppose that $c\in E^r_{0,p-1}$ and $c = d^{(r)} b$, for some  $b\in E^{r}_{r,p-r}$. Now, every element of  $E^{r}_{r,p-r}$ is necessarily a combination 
\[
b = \sum _j b_{1,j}\cdot b_{2,j}\qquad 
(b_{1,j}\in E^r_{r,0},\,\,\, b_{2,j}\in E^r_{0,p-r}).
\]
Applying the differential $d^{(r)}$, and keeping in mind that $d^{(r)}$ vanishes on $E^r_{0,p-r}$, we find that 
\[
c= d^{(r)} b = \sum _j (d^{(r)}b_{1,j})\cdot  (b_{2,j}),
\]
which gives the required representation of $c$.

We now have isomorphisms
\[
\Indecomp (E^{2}_{0,p-1}) \stackrel\cong \to \Indecomp(E^{3}_{0,p-1})\stackrel\cong\to \cdots \stackrel\cong\to \Indecomp (E^{p}_{0,p-1})  
\]
induced from the quotient maps from $E^{r}_{0,p-1}$ to $E^{r+1}_{0,p-1}$.  The proof is concluded by observing that the quotient map 
\[
 E^{p}_{0,p-1}\longrightarrow   \Indecomp (E^{p}_{0,p-1}) 
\]
is in fact an isomorphism, since 
\[
0< s < p-1\quad \Rightarrow \quad  E^{p}_{0,s} = 0,
\]
thanks to the assumed vanishing of $H_p(\mathfrak{g},\C)$ for all $p>0$, and thanks to the fact  that the spaces $E^p_{0,s}$ above  are equal to their $E^\infty$-counter\-parts. 
\end{proof}

\begin{theorem} 
\label{thm-primitive-element-theorem}
Let $(\mathfrak{g},\mathfrak{h})$ be an admissible pair of Lie algebras. If $\mathfrak{g}/\mathfrak{h}$ acts trivially on   $H_*(\mathfrak{h},\C)$,  and  if 
$H_p (\mathfrak{g},\C) =0$ for all $p>0$, then there are isomorphisms 
\[
\Prim \bigl (H_{p+1}(\mathfrak{g}/ \mathfrak{h}, \C)\bigr ) 
\cong 
\Prim \bigl (H_{p}( \mathfrak{h}, \C)\bigr ) \qquad \forall p \ge 0 .
\]
\end{theorem}

\begin{proof}
The assumption that $H_1 (\mathfrak{g},\C) =0$   implies that $H_1(\mathfrak{g}/\mathfrak{h},\C)=0$  because the natural map 
\[
H_1(\mathfrak{g}, \C) \longrightarrow H_1(\mathfrak{g}/\mathfrak{h}, \C) 
\]
 is always surjective. This proves the $p=0$ case.

For   $p\ge 1$,    the assumptions   that $H_{p}(\mathfrak{g}, \C)=0$ and  $H_{p+1} (\mathfrak{g},\C) =0$ imply that  the differential 
\begin{equation}
    \label{eq-differential-on-page-p}
d^{(p+1)} \colon E^{p+1}_{p+1,0}\longrightarrow E^{p+1}_{0,p}
\end{equation}
on the $p+1$'th page of the Hochschild-Serre spectral sequence is an isomorphism. This is because there are no non-zero differentials beyond this page that begin or end at either of the places $(p+1,0)$ or $(0,p)$, and therefore  
\[
\ker(d^{(p+1)}) \stackrel{\text{def}}= E^{p+2}_{p+1,0}= E^{p+2}_{p+1,0}=\cdots =  E^{\infty}_{p+1,0},
\]
while $E^{\infty}_{p+1,0}=0$ because it is a subquotient of $H_{p+1}(\mathfrak{g} , \C)$.
Similarly
\[
\operatorname{coker}(d^{(p+1)}) \stackrel{\text{def}}= E^{p+2}_{0,p}= E^{p+3}_{0,p}=\cdots =  E^{\infty}_{0,p} = 0,
\]
 since $E^{\infty}_{0,p}$ is a subquotient of $H_{p}(\mathfrak{g} , \C)$.

For $p\ge 1$ we  can now form the diagram of isomorphisms 
\[
\xymatrix{
   &  & E^{p+1}_{p+1,0} \ar[r]^-\cong \ar[d]_{d^p}^-{\cong}& \Prim (E^2_{p+1,0})
    \\
  \Prim (E^2_{0,p}) \ar[r]_-{\cong} &  \Indecomp (E^2_{0,p}) \ar[r]_-{\cong}  & E^{p+1}_{0,p}
}
\]
in which the left-most horizontal isomorphism is from Theorem~\ref{thm-prim-iso-to-indecomp}, the next two horizontal isomorphisms are from Lemmas~\ref{lem-side-edge-terms-in-spectral-sequence} and   \ref{lem-bottom-edge-terms-in-spectral-sequence}, respectively,   and the vertical isomorphism is the map \eqref{eq-differential-on-page-p} that we have studied above.
\end{proof}

\section{Locally generated derivations for a one-point space} 
\label{sec-homology-of-lgd-for-point}

We shall compute homology for the  Lie algebra of locally generated derivations in the case of a one-point space.

\subsection{Description of the Lie algebra}

In the case of a one-point space $\{ x\}$, since every derivation of the matrix algebra  $A_x\cong M_{2^\alpha} (\C)$ is locally generated, 
 \[
    \mathfrak{lgd}_\alpha (\mathrm{pt}) \cong \mathfrak{pgl}(n,\C)\qquad (n= 2^{\alpha(x)}).
\]
Of course, the projective general linear Lie algebra $\mathfrak{pgl}(n,\C)$ is isomorphic to $\mathfrak{sl}(n,\C)$.  We find therefore that 
 \[
    \lgd (\mathrm{pt}) \cong \varinjlim _k \mathfrak{sl}(2^k, \C) ,
\] 
where the morphisms in the directed system are 
\[
\begin{gathered}
\mathfrak{sl}(2^k, \C) \longrightarrow \mathfrak{sl}(2^{k+1}, \C)
\\
X \longmapsto 
\left [\begin{smallmatrix}  X & 0 \\ 0 & X 
\end{smallmatrix}\right ]
\end{gathered}
\]

As for the subalgebra
$\lgd^0_\alpha (\mathrm{pt})\subseteq \lgd_\alpha(\mathrm{pt})$, using an isomorphism 
\[
A_\alpha(x) \cong M_n(\C) \qquad (n= 2^{\alpha(x)})
\]
under which the state  $\varphi\colon  M_n(\C) \to \C$ has the form
\[
\varphi ([t_{ij}])= t_{11}
\]
we find that 
\[
\lgd^0_\alpha (\mathrm{pt}) \cong 
\Bigl\{ 
\left [\begin{smallmatrix}  x_{11} & 0 & \cdots & 0
\\
0 & x_{22} & \cdots& x_{2n}
\\
\vdots &\vdots &&\vdots 
\\ 0 & x_{n2}& \cdots & x_{nn} 
\end{smallmatrix}\right ]
\Bigr \} \Big / \operatorname{Center} \bigl ( \mathfrak{gl}(n,\C)\bigr ) \cong  \mathfrak{gl}(n{-}1,\C) ,
\]
via the map
\[
\left [\begin{smallmatrix}  x_{11} & 0 & \cdots & 0
\\
0 & x_{22} & \cdots& x_{2n}
\\
\vdots &\vdots &&\vdots 
\\ 0 & x_{n2}& \cdots & x_{nn} 
\end{smallmatrix}\right ]\longmapsto 
\left [\begin{smallmatrix}  
  x_{22}-x_{11} & \cdots& x_{2n}
\\
 \vdots &&\vdots 
\\   x_{n2}& \cdots & x_{nn}-x_{11} 
\end{smallmatrix}\right ].
\]
It follows that 
\[
\lgd^0(\mathrm{pt}) \cong \varinjlim_k    \mathfrak{gl}(2^k{-}1,\C)
\]
under the directed system of morphisms
\[
\begin{gathered}
 \mathfrak{gl}(2^k{-}1, \C) \longrightarrow  \mathfrak{gl}(2^{k+1}{-}1, \C)
\\
X \longmapsto 
 \left [\begin{smallmatrix}  X & 0 &0\\ 0& 0 & 0 \\ 0& 0 & X 
\end{smallmatrix}\right ] .
\end{gathered}
\]
\subsection{Computation for the unitary group}
In order to compute homology, we may use some well-known results involving Lie groups.  We start with 
   \[
   \mathfrak{su}(n)\otimes \C \cong \mathfrak{sl}(n,\C)
   \]
from which it follows that $H_*(\mathfrak{sl}(n,\C),\C)$ is the complexification of $H_*(\mathfrak{su}(n),\R)$.

Now if $\mathfrak{g}$ is the Lie algebra of a finite-dimensional, compact Lie group, then it follows from an averaging argument that  the homology of $\mathfrak{g}$ with trivial coefficients may be computed from the subcomplex
\begin{equation}
    \label{eq-complex-for-lie-algebras-of-compact-groups}
\xymatrix{
  \bigl [\Lambda^0(\mathfrak{g})\bigr ]^G   &  \bigl [ \Lambda^1(\mathfrak{g}) \bigr ]^G \ar[l]_{b'} &  \bigl [\Lambda^2(\mathfrak{g}) \bigr ]^G \ar[l]_{b'} &\cdots  \ar[l]_-{b'}
}
\end{equation}
of the complex \eqref{eq-complex-for-homology-with-trivial-coefficients}.  But actually \emph{the differentials in \eqref{eq-complex-for-lie-algebras-of-compact-groups} are all zero}, and so we obtain isomorphisms 
\begin{equation}
\label{eq-chevalley-eilenberg-formula}
\bigl [\Lambda^p(\mathfrak{g})\bigr ]^G \stackrel \cong \longrightarrow H_p (\mathfrak{g}, \C)
\end{equation}
when $\mathfrak{g}$ is the Lie algebra of a compact group. This isomorphism  is due to Chevalley and Eilenberg \cite{ChevalleyEilenberg48}.

Hopf \cite{Hopf41} established the Milnor-Moore-type isomorphism
\[
H_* (\mathfrak{g},\C) \cong  
\Lambda^* \bigl ( \operatorname{Prim}  (H_* (\mathfrak{g}, \C) )\bigr )  
\]
in this case (long before the work of Milnor and Moore, of course), which makes it of interest to determine the primitive elements in Lie algebra homology.  

When $\mathfrak{g} = \mathfrak{u}(n)$, the explicit determination of those elements in $[\Lambda ^* (\mathfrak{g})]^G$ that correspond to the primitive elements in Lie algebra homology under the isomorphism \eqref{eq-chevalley-eilenberg-formula} is due to Dynkin \cite{Dynkin53translated,Dynkin54translated}; our formula \eqref{eq-dynkin-formula} is \cite[Eqn~(8.8)]{Dynkin54translated}, up to normalization.
See also Kostant's work on this topic, \cite{Kostant58}. Using the normalized invariant inner product 
\[
\langle X,Y\rangle = \frac{1}{n} \operatorname{Re} \left ( \Tr (X^*Y)\right ) 
\]
on  $\mathfrak{u}(n)$ to identify $\mathfrak{u}(n)$ with its vector space dual, and using the associated isomorphisms
\[
\Lambda ^* \mathfrak{u}(n) \cong \Lambda ^* \mathfrak{u}(n)^* \cong \left [\Lambda ^* \mathfrak{u}(n)\right ]^*,
\]
the primitive elements are precisely  multiples of the forms
\begin{equation}
    \label{eq-dynkin-formula}
X_1\wedge \cdots \wedge X_p \longmapsto \sum _{\sigma} (-1)^\sigma \frac 1 n \Tr(X_{\sigma(1)}\cdots X_{\sigma(p)}),
\end{equation}
where $p =1,3,5,\dots ,2n{-}1$, and where the sums are over all permutations. 

The scalar normalizations are chosen  so that under the inclusions 
\[
\begin{gathered}
\mathfrak{u}(n)\longrightarrow \mathfrak{u}(2n)
\\
X \longmapsto \left [\begin{smallmatrix}  X & 0 \\ 0 & X 
\end{smallmatrix}\right ] ,
\end{gathered}
\]
the given primitive generators of $\mathfrak{u}(n)$ map to the same for $\mathfrak{u}(2n)$. The following result follows  from  this:
 
\begin{theorem}
\label{thm-homology-of-u-2-k}
  \[
  \operatorname{Prim}  \bigl (H_p (\varinjlim _{k} \mathfrak{u}(2^k), \R) \bigr ) =
  \begin{cases} \R & p = 1,3,5,7,\dots \\ 0 & \text{otherwise}.\end{cases}
  \]
\end{theorem}

Now $\mathfrak{u}(n) \cong \R \times \mathfrak{su}(n)$, and it follows from the above that  
 \[
  \operatorname{Prim}  \bigl (H_p (\varinjlim _{k} \mathfrak{su}(2^k), \R) \bigr ) =
  \begin{cases} \R & p = 3,5,7,\dots \\ 0 & \text{otherwise},\end{cases}
  \]
from which we obtain:
 
\begin{theorem}
\label{thm-homology-of-lgd-of-pt}
  \[
  \operatorname{Prim}  \bigl(H_p (\mathfrak{lgd}(\mathrm{pt}), \C) \bigr )
  =
  \begin{cases} \C & p = 3,5,\dots \\ 0 & \text{otherwise}.\end{cases}
  \]
\end{theorem}

There is a similar computation for $\lgd^0(\mathrm{pt})$:
\begin{theorem}
\label{thm-homology-of-lgd-zero-of-pt}
  \[
  \operatorname{Prim}  \bigl(H_p (\mathfrak{lgd}^0(\mathrm{pt}), \C) \bigr )
  =
  \begin{cases} \C & p = 1,3,5,\dots \\ 0 & \text{otherwise}.\end{cases}
  \]
\end{theorem}

\section{Some derivations and endomorphisms of the Lie algebra of locally generated derivations}
\label{sec-inner-automorphisms}

The purpose of this section is to determine some sufficient conditions for an endomorphism of $\lgd(X)$ to induce the identity morphism on homology.  We shall not mention it explicitly, but exactly the same conditions apply to the subalgebra $\lgd^0(X)$ defined in Section~\ref{sec-subalgebra-fixing-a-state}.

\subsection{Inner automorphisms associated to locally nilpotent derivations}
\label{sec-automorphisms-associated-to-locally-nilpotent-derivations}
Let $X$ be a proper metric space and let $\alpha\colon X \to \N$ be any function.
Suppose given nilpotent elements 
\[
N_x\in A_{\alpha(x)}\qquad (x\in X)
\]
of uniformly bounded nilpotency order, meaning  that
\begin{equation}
    \label{eq-uniformly-bounded-nilpotency-order}
 \exists k\ge 2: N_x^k=0 \qquad \forall x\in X.
\end{equation}
For every $x\in X$, the invertible element $\exp (N_x)\in A_{\alpha(x)}$ acts by conjugation on every $A_\alpha (F)$ for which the finite subset $F\subseteq X$ includes $x$.  It also acts by conjugation on the direct limit $A_\alpha (X)$.  The automorphisms 
\[
\Ad _{t \exp(N_x)}\colon A_\alpha (X)\longrightarrow A_\alpha (X)\qquad (t\in \R,\,\,\, x\in X)
\]
commute with one another, as $t$ and $x$ vary, and if $T$ is any element of $A_\alpha(X)$, then 
\begin{equation}
    \label{eq-infinitely-many-inner-automorphisms}
\Ad _{\exp(t N_x)}(T) = T \quad \text{for all but finitely many $x\in X$}.
\end{equation}
So for any $t\in \R$ the infinite product 
\begin{equation}
    \label{eq-infinite-product-of-inner-automorphisms}
\sigma_t = \prod _{x\in X} \Ad _{\exp(t N_x)}\colon A_\alpha (X)\longrightarrow A_\alpha (X)
\end{equation}
makes sense as an automorphism of $A_\alpha(X)$.  Moreover by virtue of \eqref{eq-infinitely-many-inner-automorphisms}, the derivative 
\[
\varepsilon (T) = \frac{d}{dt}\Big \vert_{t=0} \sigma_t(T)
\]
exists inside any of the finite-dimensional algebras $A_\alpha (F)$ that contains $T$, and defines a derivation of $A_\alpha (X)$.  It is an element of the Lie algebra $\lgd_\alpha(X)$: 
\begin{equation} 
\label{eq-def-of-delta-sigma}
\varepsilon (T) = \sum_{x\in X} [N_x,T].
\end{equation}

The automorphisms $\sigma_t$ in \eqref{eq-infinite-product-of-inner-automorphisms} can be made to act as Lie algebra automorphisms 
\begin{equation}
    \label{eq-infinite-product-of-inner-automorphisms-2}
\sigma_t \colon \lgd_\alpha (X) \longrightarrow \lgd_\alpha (X)
\end{equation}
using the formula 
\[
\sigma_t (\delta)(T) = \sigma_t (\delta(\sigma_{-t}(T))\qquad (\delta \in \lgd_\alpha(X),\,\,\, T \in A_\alpha (X)),
\]
which is equivalent to the   prescription  
\[
\delta (T) = \sum_{x\in X} [H_x,T] \quad \Rightarrow \quad 
\sigma_t(\delta) (T) = \sum_{x\in X} [\sigma_t(H_x),T]
\]
(the latter formula shows that indeed $\sigma_t$ maps $\lgd_\alpha(X)$ to itself). Moreover 
\[
 \frac{d}{dt}\Big \vert_{t=0} \sigma_t(\delta) =[\varepsilon,\delta]
\]
in the sense that 
\[
 \frac{d}{dt}\Big \vert_{t=0} \sigma_t(\delta)(T) = [\varepsilon,\delta](T)\quad \forall T \in A_\alpha (X);
\]
note that the derivative may be computed in a finite-dimensional subalgebra $A_\alpha (F)$ that depends on $\delta$ and $T$.

\begin{definition} 
A proper discrete metric space $X$ is  \emph{uniformly properly discrete}  if 
    \[
    \forall R \ge 0 : \quad \sup\,\bigl \{\, \operatorname{Cardinality}\bigl ( B (x,R)\bigr ) : x \in X \,\bigr \}< \infty .
    \]
\end{definition} 

\begin{lemma} 
If $X$ is uniformly properly discrete, then the Lie algebra element 
$\varepsilon\in \lgd (X)$ in \eqref{eq-def-of-delta-sigma} is locally nilpotent.
\end{lemma}

\begin{proof}
If $\delta \in \lgd (X)$, and $\delta = \sum_{x \in X} \ H_x$, as in Definition~\ref{def-locally-generated-derivation}, with $H_x \in  A_{\alpha}(B_X(x,R))$ for all $x\in X$, then 
\begin{equation}
    \label{eq-ad-epsilon-of-delta}
ad_{\varepsilon}(\delta) = 
\sum_{x,y \in X} \ ad_{[N_y,H_x]},  
\end{equation}
where the sum becomes finite when applied to any element of $A(X)$.  Now we note 
these properties:
\begin{enumerate}[\rm (i)]
\item
$[N_{y_1}, N_{y_2}]= 0$ for all $y_1,y_2\in X$.

\item $[N_y,H_x]\in A_\alpha (B_X(x,R))$ if $y\in B_X(x,R)$

\item
$
[N_y, H_x] =
0 $ if $ y \notin B_X(x,R)$.

\end{enumerate}
With these, we may rewrite  \eqref{eq-ad-epsilon-of-delta} as
\[
ad_{\varepsilon}(\delta) = \sum_{x\in X}  \sum_{y \in \text{Ball}_X(R,x)}
\ ad_{[N_y,H_x]}
\]
and then iterate:
\begin{equation}
    \label{eq-ad-epsilon-of-delta-iterated}
ad_{\varepsilon}^m(\delta) = \sum_{x\in X}  \sum_{y_1, \cdots, y_m \in  B_X(x,R)}
\ ad_{[N_{y_1} , [N_{y_2},[ \cdots, [N_{y_m} , H_x]\cdots ]]]}
\end{equation}
Now 
\begin{equation*}
[N_{y_1} , [N_{y_2},[ \cdots, [N_{y_m} , H_x]\cdots ]]]  = 
  \sum_{ I\sqcup J = \{ 1,\dots, m\}} \pm N_I H_x N_J,
\end{equation*}
where the sum is over partitions of $\{ 1,\dots, m\}$, and where 
\[
N_I = \prod _{i\in I} N_{y_i}\quad \text{and} \quad 
N_J =  \prod _{j\in J} N_{y_j} .
\]
Since $y_r\in B_X(x,R)$ for all $r$, and since 
the cardinality of $B_X(x,R)$ is bounded
 independently  of $x \in X$,
 it follows that if $m$ is sufficiently large, then independently of the choice of  $x$,
  at least one of $N_I$ or $ N_J$ must contain a power
 $N_{y_r}^k$, so that by \eqref{eq-uniformly-bounded-nilpotency-order}, $N_I=0$ or $N_J=0$. So for large enough $m$, every term on the right-hand side of  \eqref{eq-ad-epsilon-of-delta-iterated} is zero.
 \end{proof}

From the above and Lemma~\ref{lem-inner-automorphisms-from-locally-nilpotent-elements} we obtain:

\begin{lemma}
    \label{lem-trivial-action-of-some-inner-autos}
    If the metric space $X$ is uniformly properly discrete, 
    then under the assumption \eqref{eq-uniformly-bounded-nilpotency-order}, the automorphisms $\sigma_t$ in \eqref{eq-infinite-product-of-inner-automorphisms-2} above induce the identity map on $H_*(\lgd _\alpha (X), \C)$. \qed
\end{lemma}

\subsection{Flip automorphisms}
\label{sec-flip-automorphisms}
Let $X$ be a uniformly properly discrete metric space, and let $\alpha\colon X \to \N$ be any function. Denote by 
\[
\sigma_x\colon A_{2\alpha(x)} \longrightarrow A_{2\alpha(x)}
\]
the involutive algebra automorphism  that exchanges the first $\alpha(x)$ tensor factors in $A_{2\alpha(x)}$ with the last $\alpha (x)$ tensor factors: 
\begin{multline*}
\sigma_x \colon T_1\otimes \cdots \otimes T_{\alpha(x)}\otimes T_{\alpha(x)+1}\otimes \cdots \otimes T_{2\alpha(x)}
\\
\longmapsto 
 T_{\alpha(x)+1}\otimes \cdots \otimes T_{2\alpha(x)} \otimes  T_1\otimes \cdots \otimes T_{\alpha(x)}.
\end{multline*}
We may regard each $\sigma_x$ as an automorphism of $A_{2 \alpha }(X)$, acting on only the $x$-factor of the infinite tensor product.  When viewed in this way,   all the $\sigma_x$ are involutive algebra automorphisms, they all commute with one another, and all but finitely many act trivially on any given $T \in A_{2 \alpha}(X)$. So we may unambiguously form the combined automorphism
\[
\sigma = \prod_{x\in X} \sigma_x \colon A_{2 \alpha }(X) \longrightarrow A_{2 \alpha }(X).
\]
There is an induced Lie algebra automorphism
\begin{equation}
    \label{eq-flip-on-lgd}
\sigma \colon \lgd_{2\alpha} (X) \longrightarrow \lgd _{2 \alpha} (X)
\end{equation}
that is defined by either of the equivalent formulas 
\[
\sigma(\delta)(T) = \sigma(\delta(\sigma^{-1}(T))
\]
or 
\begin{equation}
    \label{eq-def-of-sigma-of-delta}
\sigma(\delta)(T) = \sum_{x\in X} [ \sigma(H_x), T]
\quad \text{if} \quad \delta(T) =  \sum_{x\in X} [  H_x, T].
\end{equation}

\begin{theorem}
\label{thm-flip-is-identity-on-homology}
Let $X$ be a uniformly properly discrete metric space, and let $\alpha\colon X \to \N$ be any function such that $\alpha (x)\ge 2 $ for all $x\in X$.
    The flip automorphism $\sigma$ in \eqref{eq-flip-on-lgd} induces the identity map on homology: 
    \[
    \sigma_*= \mathrm{id} \colon H_*(\lgd_{2\alpha}(X),\C) \longrightarrow H_*(\lgd_{2\alpha}(X),\C) .
    \]
\end{theorem}

\begin{proof}
Fix an isomorphism of algebras $A_{2n} \cong M_{2^{2n}}(\C)$ so as to be able to consider the flip automorphism
\[
\sigma_x \colon A_{2\alpha(x)} \longrightarrow A_{2\alpha(x)}
\]
as an automorphism
\[
\sigma_n \colon M_{2^{2n}}(\C) \longrightarrow M_{2^{2n}}(\C) ,
\]
where $n=\alpha(x)$.  The automorphism $\sigma_n$ is given by conjugation with a self-adjoint unitary matrix $S$ that has eigenvalues $\lambda = \pm 1$ with multiplicities
\[
\begin{cases} 
2^{n-1}(2^{n}{+}1) & \lambda =1
\\
2^{n-1}(2^{n}{-}1) & \lambda =-1
\end{cases}
\]
Assuming that $n\ge 2$ both multiplicities, and in particular the multiplicity for $\lambda = -1$, are even.  Decomposing the eigenspace for $\lambda = -1$ into a direct sum of $2$-dimensional subspaces, and using the matrix factorization
\[
\begin{bmatrix} 
-1 & 0 
\\
0 & -1 
\end{bmatrix}
=
\left ( \begin{bmatrix} 
1 & 1
\\
0& 1
\end{bmatrix}
\cdot 
\begin{bmatrix} 
\phantom{.}1 & \phantom{.}0
\\
-1 & \phantom{.}1
\end{bmatrix}
\cdot 
\begin{bmatrix} 
1 & 1
\\
0 &1
\end{bmatrix}
\right )^2 ,
\]
 we  may write 
\[
S = \exp (N_{1})\exp (N_{2})\cdots \exp (N_{6})
\]
where each $N_{j}\in M_{2^{2n}}(\C)$ has square zero.  

It follows that each flip automorphism $\sigma_x$  has the form 
\[
\sigma_x = \Ad_{\exp (N_{1,x})}\Ad_{\exp (N_{2,x})}\cdots \Ad_{\exp (N_{6,x})}
\]
where all $N_{j,x}\in A_{2 \alpha (x)}$ have square zero.  We find therefore that the global flip automorphism 
\[
\sigma\colon A_{2\alpha} (X) \longrightarrow A_{2\alpha}(X)
\]
may be written as a composition of $6$ automorphisms, each of precisely the form considered in Section~\ref{sec-automorphisms-associated-to-locally-nilpotent-derivations}.  The theorem therefore follows from Lemma~\ref{lem-trivial-action-of-some-inner-autos}.
\end{proof}

\subsection{Triviality of the flip automorphism on the homology of ideals}
\label{sec-triviality-of-lgd-action-on-homology}

\begin{theorem}
\label{thm-flip-is-identity-on-homology-in-realtive-case}
Let $Z$ be a uniformly properly discrete metric space, let $Y$ be any subspace, and let $\alpha\colon Z \to \N$ be any function such that $\alpha (x)\ge 2 $ for all $x\in Z$.
    The flip automorphism $\sigma$ in \eqref{eq-flip-on-lgd} induces the identity map on homology: 
    \[
    \sigma_*= \mathrm{id} \colon H_*(\lgd_{2\alpha}(Z,Y),\C) \longrightarrow H_*(\lgd_{2\alpha}(Z,Y),\C) .
    \]
\end{theorem}

\begin{proof}
    Lemma~\ref{lem-lgd-ideal-is-increasing-union} reduces the theorem to the same assertion for the Lie algebras $ \lgd _\alpha (\operatorname{Pen}_Z(Y,R))$, proved in the previous section.
\end{proof}

In Section~\ref{sec-homology-of-lgd-for-n-space} we shall also need the following assertion about the ideals $\lgd(Z,Y)$,  when we apply our primitive element theorem to Lie algebras of locally generated derivations.

\begin{theorem}
\label{thm-trivial-action-of-lgd-on-homology}
    If $Z$ is any proper discrete metric space,  if $Y$ is a subspace of $Z$, and if $\alpha\colon Z{\to} \N$ is any function, then $\lgd_\alpha(Z)$ acts trivially on $H_*(\lgd_\alpha (Z,Y),\C)$.
\end{theorem}

To prove this we shall use the following:

\begin{lemma}
\label{lem-sum-decomposition-of-lgd}
 If $Z$ is any proper discrete metric space,  if $V$ and $W$  are subspace of $Z$ with $Z = V \cup W  $, and if $\alpha\colon Z{\to} \N$ is any function, and if $S> 0$ then
    \[
    \lgd_{\alpha, S}(Z) = \lgd_{\alpha,S}(V) + \lgd_{\alpha,S}(\operatorname{Pen}_Z(W;2S))
    \]
    \textup{(}notation from Definition~\textup{\ref{def-locally-generated-derivation}}\textup{)}.
\end{lemma}

\begin{proof}
    If $\delta = \sum_{z\in Z} \ad_{H_z}$ is a derivation in $\lgd_{\alpha, S}(Z)$, and if we write 
    \[
    \delta = \sum_{z\notin \operatorname{Pen}_Z(W;S)} \ad_{H_z} + \sum_{z\in  \operatorname{Pen}_Z(W;S)} \ad_{H_z},
    \]
    then the first sum is an element of  $\lgd_{\alpha,S}(V)$ and the second is an element of  $\lgd_{\alpha,S}(\operatorname{Pen}_Z(W;2S))$.
\end{proof}

\begin{proof}[Proof of Theorem~\ref{thm-trivial-action-of-lgd-on-homology}]
Let $\delta\in \lgd_{\alpha, S}(Z)$. In view of Lemma~\ref{lem-lgd-ideal-is-increasing-union}, it suffices to show that $\delta$ acts trivially on any class  $[c]\in H_*(\lgd_{\alpha}(Z,Y))$ that is represented by a cycle $c$ in any of the subcomplexes
\[
\Lambda^*(\lgd_{\alpha}(\operatorname{Pen}_Z(Y;R)))
\subseteq \Lambda^*(\lgd_{\alpha}(Z,Y))\qquad (R>0).
\]
Having fixed one of these subcomplexes, write 
\[
Z = \bigl (Z \setminus \operatorname{Pen}_Z(Y;R)\bigr) \cup
\operatorname{Pen}_Z(Y;R)
\]
and apply Lemma~\ref{lem-sum-decomposition-of-lgd} to obtain a decomposition
\[
\delta = \delta_1 + \delta_2 \in  \lgd_{\alpha,S}(Z \setminus \operatorname{Pen}_Z(Y;R)) + \lgd_{\alpha,S}(\operatorname{Pen}_Z(Y;R{+}2S)).
\]
The derivation $\delta_1$ acts trivially on the  $\Lambda^*(\lgd_{\alpha}(\operatorname{Pen}_Z(Y;R)))$ and in particular on the cycle  $c$. The derivation $\delta_2$ acts as an inner derivation on the intermediate subcomplex
\[
\Lambda^*(\lgd_{\alpha}(\operatorname{Pen}_Z(Y;R)))
\subseteq \Lambda^*(\lgd_{\alpha}(\operatorname{Pen}_Z(Y;R{+}2S)))
\subseteq \Lambda^*(\lgd_{\alpha}(Z,Y))
\]
and so it  acts trivially on the homology class $[c]$.
\end{proof}

\subsection{Using the flip automorphism} 
\label{sec-using-the-flip-automorphism}
We shall show in this section how Theorem~\ref{thm-flip-is-identity-on-homology} may be used to compute the action of many more Lie algebra morphisms on homology.  We shall formulate our result for  $\lgd_\alpha (X)$, but it applies equally well to the ideals $\lgd_\alpha (X,Y)$.

Let $X$ be a uniformly properly discrete metric space,  let $\alpha, \beta \colon X \to \N$ be  functions, and let 
\begin{equation*}
\tau \colon A_\alpha (X) \longrightarrow A_\beta (X) 
\end{equation*}
be any unital algebra morphism with the following property:
\begin{equation}
\label{eq-local-property-of-algebra-morphism-tau}
    \exists R > 0:\quad \tau [ A_\alpha (F)]\subseteq  A_\beta (\operatorname{Pen}_X(F;R)) \qquad \forall F \subseteq X
\end{equation}
(we shall use this property for all finite $F$, but this implies the property for all $F$, finite or not).  Associated to the Lie algebra morphism $\tau$ there is a Lie algebra morphism
\begin{equation}
\label{eq-lie-algebra-morphism-tau}
\tau \colon \lgd _\alpha (X) \longrightarrow \lgd _\beta (X)
\end{equation}
that is characterized by the formula 
\begin{equation}
\label{eq-definition-of-lie-algebra-morphism-tau}
\tau(\delta)(T) = \sum_{x\in X} [ \tau(H_x), T]\qquad \forall T \in A_\beta (X),
\end{equation}
 when $\delta(T) = \sum_{x\in X} [ H_x, T]$ for all $T\in A_\alpha(X)$, as in Definition~\ref{def-locally-generated-derivation}.  

 \begin{theorem}
 \label{thm-using-the-flip}
Let $X$ be a uniformly properly discrete metric space,  let $\alpha, \beta \colon X \to \N$ be  functions, and let $\tau \colon A_\alpha (X) \to A_\beta (X) $ be any unital algebra morphism satisfying \eqref{eq-local-property-of-algebra-morphism-tau}.  The composition 
\[
\xymatrix{
 \lgd_\alpha(X)  \ar[r]^{\tau} &  \lgd_\beta(X) \ar[r]^{\iota_{2\beta,\beta}} &   \lgd_{2\beta}(X). 
}
\]
involving  the associated Lie algebra morphism \eqref{eq-lie-algebra-morphism-tau} induces the same map in homology with trivial coefficeints, $\C$, as the standard inclusion morphism
 \[
\xymatrix@C=30pt{
 \lgd_\alpha(X)  \ar[r]^{\iota_{2\beta,\alpha,*}} &   \lgd_{2\beta}(X).
}
\]
 \end{theorem}

 \begin{proof} 
 The isomorphisms
 \[
 A_{\alpha(x)}\otimes A_{\alpha (x)} \stackrel \cong \longrightarrow A_{2\alpha(x)}
\]
given by the formulas 
\begin{multline*}
\bigl (S_1\otimes \dots \otimes S_{\alpha(x)}\bigr )
\otimes
\bigl (T_1\otimes \dots \otimes T_{\alpha(x)}\bigr )
\\ 
\mapsto 
S_1\otimes \dots \otimes S_{\alpha(x)}\otimes T_1\otimes \dots \otimes T_{\alpha(x)},
\end{multline*}
and the similar isomorphisms for $\beta$ in place of $\alpha$, determine   isomorphisms
\[
A_\alpha (X) \otimes A_\alpha (X) \stackrel \cong \longrightarrow  A_{2\alpha} (X)\quad \text{and} \quad 
A_\beta (X) \otimes A_\beta (X) \stackrel \cong \longrightarrow  A_{2\beta} (X)
\]
Using these, we may define associative algebra morphisms
\[
\begin{aligned}
\tau_{\mathrm{left}}&=\tau \otimes \mathrm{id}\colon A_{2\alpha }(X)\longrightarrow A_{2 \beta}(X) 
\\
\tau_{\mathrm{right}}& = \mathrm{id}\otimes \tau 
\colon A_{2\alpha }(X)\longrightarrow A_{2 \beta}(X).
\end{aligned}
\]
These satisfy the property \eqref{eq-local-property-of-algebra-morphism-tau} and so induce morphisms of Lie algebras of locally generated derivations. The latter fit into commuting diagrams 
\[
\xymatrix{
 \lgd_\alpha(X) \ar@{=}[d] \ar[r]^{\iota_{2 \alpha,\alpha}} & \lgd_{2\alpha}(X)\ar[r]^{\tau_{\mathrm{left}}} &  \lgd_{2\beta}(X) \ar@{=}[d]
 \\
 \lgd_\alpha(X) \ar[r]_{\tau} & \lgd_{\beta}(X)  \ar[r]_{\iota_{2\beta, \beta}} & \lgd_{2 \beta}(X)
}
\] 
and 
\[
\xymatrix{
 \lgd_\alpha(X) \ar@{=}[d] \ar[r]^{\iota_{2 \alpha,\alpha}} & \lgd_{2\alpha}(X)  \ar[r]^{\tau_{\mathrm{right}}} &  \lgd_{2\beta}(X) \ar@{=}[d]
 \\
 \lgd_\alpha(X) \ar[r]_{\iota_{\beta,\alpha}} & \lgd_{\beta}(X)  \ar[r]_{\iota_{2\beta, \beta}} & \lgd_{2 \beta}(X)
}
\] 
that are obtained from similar diagrams of associative algbera morphisms.  But $\tau_{\mathrm{left}}$ and $\tau_{\mathrm{right}}$ induce the same morphism on Lie algebra homology with trivial coefficients in view of Theorem~\ref{thm-flip-is-identity-on-homology} and the commuting diagram
\[
\xymatrix{
   \lgd_{2\beta}(X)  \ar[r]^{\tau_{\mathrm{right}}}\ar[d]_{\sigma} &  \lgd_{2\beta}(X) \ar@{=}[d]
 \\
 \lgd_{2\beta}(X)  \ar[r]_{\tau_{\mathrm{left}}} & \lgd_{2 \beta}(X),
}
\]
in which  $\sigma$ is the flip automorphism.
 \end{proof}

\subsection{Correcting morphisms}
\label{sec-correcting-morphisms}
Let $X$ be a uniformly properly discrete metric space, and let $\alpha\colon X \to \N$ be any function.  We are going to put a Hopf algebra structure on the homology of $\lgd (X)$, using the method of product morphisms and correcting morphisms  described in Sections~\ref{subsec-algebra-structure} and \ref{sec-hopf-algebra-structure-from-product-morphism}. In this section we shall describe an appropriate family of correcting morphisms.

Let $\gamma \colon X\times \N \to \N$ be a function with the property that 
\[
\forall x \in X : \gamma (x,m) = \gamma (x,n) \quad \Rightarrow \quad m=n.
\]  
Thus $\gamma$ is 1-1 in the $\N$-variable.

\begin{remark}
\label{rem-even-and-odd-maps}
We are mostly interested in the case where $\gamma$ is independent of $x\in X$, and  good examples of this type to consider are the functions $\gamma_{\mathrm{even}}(n) = 2n$ and $\gamma_{\mathrm{odd}}(n) = 2n{+}1$. But it is convenient to consider the more general maps above, that are also functions of $x\in X$,  at the same time.
\end{remark}

Given a map $\gamma$ as above, and given $\alpha \colon X \to \N$, assume that $\beta\colon X \to \N$ has the property that 
\begin{equation}
\label{eq-condition-for-gamma-embedding}
\forall x \in X : \quad k \le \alpha (x) \quad \Rightarrow \quad \gamma (x,k) \le \beta(x)  . 
\end{equation}
Define   embeddings 
\begin{equation}
    \label{eq-gamma-x-embeddings}
\begin{gathered}
\gamma_x \colon A_{\alpha(x)} \longrightarrow A _{\beta(x)}
\\
A_{\alpha(x)}\ni T_1\otimes \cdots \otimes T_{\alpha(x)}
\longmapsto S_1\otimes \cdots \otimes S_{\beta(x)} \in A_{\beta(x)},
\end{gathered}
\end{equation}
where 
\[
S_k = 
\begin{cases} 
T_\ell & \exists \ell\in \{ 0,\dots, \alpha(x)\} :\,\,\gamma (x,\ell) = k 
\\
I & \nexists \ell\in \{ 0,\dots, \alpha(x)\} :\,\,\gamma (x,\ell) = k.
\end{cases}
\]
So each $T_\ell$ appears exactly once among the $S_k$, and all the rest of the $S_k$, meaning those that not equal to some $T_\ell$,  are equal to $I$.

The embeddings \eqref{eq-gamma-x-embeddings} determine embeddings 
\[
\gamma_F = \otimes _{x\in F} \colon A_\alpha (F) \longrightarrow A_{\beta} (F)
\]
for every finite subset of $F$.  These are compatible with inclusions of finite subsets, in the sense that if $F'\subseteq F''$, then the diagram
\[
\xymatrix{
A_\alpha (F') \ar[r]^{\gamma_{F'}} \ar[d] & A_{\beta}(F') \ar[d]
\\
A_\alpha (F'') \ar[r]_{\gamma_{F''}} & A_{\beta}(F'')
}
\]
commutes, where the vertical morphisms are as in \eqref{eq-embed-a-f-prime-into-a-f-double-prime}.  There is therefore an induced embedding of direct limits 
\[
\gamma_X \colon A_\alpha (X) \longrightarrow A_{\beta}(X) .
\]
Moreover this morphism is compatible with our order relation \eqref{eq-order-on-alphas}, in the sense that if $\alpha \le \beta_1 \le \beta_2$, then the diagram
\[
\xymatrix{
A_\alpha (X) \ar[r]^{\gamma_X}  \ar@{=}[d] & A_{\beta_1}(X) \ar[d]^{\iota_{\beta_2,\beta_1}}
\\
A_\alpha (X) \ar[r]_{\gamma_X} & A_{\beta_2}(X)
}
\]
commutes (here the morphism $\iota_{\beta_2,\beta_1}\colon A_{\beta_1}(X) \to A_{\beta_2}(X)$ is the canonical embedding from \eqref{eq-embedding-of-lgd-alpha-into-lgd-beta}). So passing to a second direct limit, we obtain from the  $\gamma$ with which we started an induced morphism of associative algebras
\begin{equation}
\label{eq-gamma-morphism-on-a-x}
\gamma_X \colon A(X)\longrightarrow A(X).
\end{equation}
It is not necessarily an automorphism. 
Once again, there is an associated morphism of Lie algebras 
\begin{equation}
\label{eq-gamma-morphism-on-lgd-x}
\gamma_X \colon \lgd(X)\longrightarrow \lgd (X),
\end{equation}
although since \eqref{eq-gamma-morphism-on-a-x} is not necessarily an automorphism, we need to use the explicit formula 
\[
\gamma(\delta)(T) = \sum_{x\in X} [ \gamma_X(H_x), T]
\quad \text{if} \quad \delta(T) =  \sum_{x\in X} [  H_x, T]
\]
to define it.

\begin{theorem} 
\label{thm-lgd-correcting-morphisms}
Let $X$ be a uniformly proper discrete metric space. The Lie algebra morphism $\gamma_X \colon \lgd (X) \to \lgd (X)$ in \eqref{eq-gamma-morphism-on-lgd-x} induces the identity map on homology: 
\[
\gamma_{X,*} =\mathrm{id} \colon H_*(\lgd (X),\C)   \longrightarrow  H_*(\lgd (X),\C) .
\]
\end{theorem}

\begin{proof}
Given any   $\alpha\colon X \to \N$, choose $\beta\colon X \to \N$ such that $\beta\ge \alpha$  and \eqref{eq-condition-for-gamma-embedding} holds, which implies  that $\gamma_X$ is well-defined as a morphism
\begin{equation*}
\gamma_X \colon \lgd_\alpha(X) \longrightarrow \lgd_\beta(X).
\end{equation*}
Theorem~\ref{thm-using-the-flip} shows that the diagram 
\[
 \xymatrix{
    \lgd_\alpha(X) \ar@{=}[d]  \ar[r]^{\gamma_X} & \lgd_{\beta}(X) \ar[r]^{\iota_{2\beta,\beta}}& \lgd_{2 \beta} (X) \ar@{=}[d]
    \\
    \lgd_{\alpha}(X) \ar[r]_{\iota_{\beta,\alpha}} & \lgd_{\beta} (X)\ar[r]_{\iota_{2\beta,\beta}} & \lgd_{2 \beta} (X)
    }
\]
commutes at the level of homology with trivial coefficients, and this gives the required result.
\end{proof}

\subsection{Product morphisms}
As in Remark~\ref{rem-even-and-odd-maps}, denote  by 
\[
\gamma_{\mathrm{even}}, \gamma_{\mathrm{odd}}\colon \N \longrightarrow \N
\]
the 1-1 maps 
\[
\gamma_{\mathrm{even}}(n) = 2n \quad \text{and} \quad \gamma_{\mathrm{odd}}(n) = 2n{+}1\qquad (n\in \N).
\]
The images of these maps are disjoint from one another, of course, and as a result the images of the maps 
\[
\gamma_{\mathrm{even}}, \gamma_{\mathrm{odd}}\colon \lgd(X) \longrightarrow \lgd(X)
\]
commute with one another.  We may therefore  define  a product morphism 
\[
\mu\colon \mathfrak{lgd}(X)\times \mathfrak{lgd}(X) \longrightarrow \mathfrak{lgd}(X)
\]
by means of the formula 
\begin{equation}
\label{eq-even-odd-definition-of-mu}
\mu(\delta_1, \delta_2) = \gamma_{\mathrm{even},X}(\delta_1) + \gamma_{\mathrm{odd},X}( \delta _2).
\end{equation}

\begin{theorem}
\label{thm-gamma-morphisms-trivial-on-homology}
    The product morphism in \eqref{eq-even-odd-definition-of-mu}, and the correcting morphisms in  Section~\textup{\ref{sec-correcting-morphisms}}, satisfy the conditions \textup{\ref{eq-first-product-morphism-condition}-\ref{eq-last-product-morphism-condition}}. Restricting these structures to the ideal $\lgd(X,Y)\triangleleft \lgd(X)$ we obtain an admissible pair as in Definition~\textup{\ref{def-admissible-pair}}. \qed
\end{theorem}

\subsection{Shift morphisms}

If $j\colon X \to X$ is any 1-1 map, then there are, for all functions $\alpha,\beta \colon X \to \N$ with $\alpha (x) \le \beta (j(x))$   induced algebra morphisms 
\[
j\colon A_\alpha(X)\longrightarrow A_{\beta}(X) .
\]
mapping $A_{\alpha(x)}$ into $A_{\beta(j(x))}$ in the standard way, described in \eqref{eq-embedding-of-matrix-algebras}.
These are compatible with limits over the directed system of all possible $\alpha$ and $\beta$, and so we   obtain an induced morphism 
\[
j\colon A(X) \longrightarrow A(X); 
\]
it is an isomorphism onto $A(j(X))$.  If in addition $j$ has the property that
\[
\sup_{x\in X} d(x,j(x))<\infty,
\]
 then there is also an induced Lie algebra morphism

\begin{equation}
    \label{eq-lie-algebra-morphism-j}
j\colon \lgd  (X)\longrightarrow \lgd (X) 
\end{equation}
for which 
\[
j(\delta)(T) = \sum_{x\in X} [ j(H_x), T]
\quad \text{if} \quad \delta(T) =  \sum_{x\in X} [  H_x, T].
\]
Compare \eqref{eq-definition-of-lie-algebra-morphism-tau}.  The following is an immediate consequence of Theorem~\ref{thm-using-the-flip}:

\begin{theorem}
\label{thm-j-induces-identity-on-homology}
Let $X$ be a uniformly proper discrete metric space and let $j\colon X \to X$ be as above.    The associated  Lie algebra morphism \eqref{eq-lie-algebra-morphism-j} induces the identity map on homology. \qed
\end{theorem}

This applies, in particular, to  the shift map
$s \colon \Z^n\to \Z^n$ defined by 
\[
s(k_1,\dots,k_{n-1},k_n) =(k_1,\dots,k_{n-1},k_n{+}1)\qquad \forall (k_1,\dots, k_n) \in \Z^n  .
\]
Let us also  write 
\begin{equation*}
\Z^n_+ = \{\, (k_1,\dots, k_n) \in \Z^n : k_n \ge 0\, \}
\end{equation*}
and
\begin{equation*}
\Z^n_- = \{\, (k_1,\dots, k_n) \in \Z^n : k_n \le 0\, \},
\end{equation*}
and consider $\Z^{n-1}$ as embedded in $\Z^n$ via the obvious isomorphism
\[
\Z^{n-1} \stackrel \cong\longrightarrow  \{\, (k_1,\dots, k_n) \in \Z^n : k_n  =0\, \}.
\]
In line with this, we shall write $\Z^0 = \{ 0\}$. 
Note that  Lie algebra  morphism 
\[
s\colon \lgd(\Z^n) \longrightarrow \lgd(\Z^n)
\]
maps the ideal $\lgd(\Z^n,\Z^n_+)$ into itself, and it also maps $\lgd(\Z^n_+)$ to itself.

\begin{theorem}
\label{thm-s-star-equals-1}
The shift maps   on $\lgd(\Z^n)$ and  $\lgd(\Z^n_+)$
  induce the identity map on homology.  \qed
\end{theorem}

On the subject of $\Z^n$ and $\Z^n_+$, we shall also use the following result in the next section:

\begin{theorem}
\label{thm-absolute-v-relative}
Let $n\ge 1$.
    The inclusions 
    \[
    \lgd(\Z^{n-1}) \to \lgd(\Z^n_{+},\Z^{n-1})
    \quad\text{and} \quad 
    \lgd(\Z_+^n) \to \lgd(\Z^n,\Z_+^n)
    \]
    induce  isomorphisms 
    \[
    H_*(\lgd (\Z^{n-1}), \C) \stackrel \cong \longrightarrow H_*(\lgd (\Z^n_+,\Z^{n-1}), \C)
   \]
   and
   \[
      H_*(\lgd (\Z_{+}^n), \C) \stackrel \cong \longrightarrow H_*(\lgd (\Z^n,\Z_{+}^n), \C).
    \]
\end{theorem}

\begin{proof}
Let us consider first the morphism 
\begin{equation}
    \label{eq-induced-morphism-in-homology}
    H_*(\lgd(\Z^{n-1}),\C) \to H_*(\lgd(\Z^n_{+},\Z^{n-1}),\C).
\end{equation}
We shall use the formula 
\[
\lgd(\Z^n_{+},\Z^{n-1}) = \bigcup _{k} \lgd (\Z^{n-1}{\times}\{ 0,\dots, k\}),
\]
from Lemma~\ref{lem-lgd-ideal-is-increasing-union},  using which  it suffices to prove that the inclusions
\begin{equation}
    \label{eq-the-morphism-j}
  j\colon \lgd  (\Z^{n-1})\longrightarrow \lgd  \bigl(\Z^{n-1}{\times}\{ 0,\dots, k\}\bigr)
\end{equation}
that are induced from the inclusions 
\[
\begin{gathered}
    \Z^{n-1} \longrightarrow \Z^{n-1}{\times}\{ 0,\dots, k\}
    \\
    (k_1,
    \dots, k_{n-1})\longmapsto (k_1,
    \dots, k_{n-1},0)
\end{gathered}
\]
give rise to isomorphisms in homology.

Given $\alpha\colon \Z^{n-1}{\times} \{ 0,\dots, k\}\to \N$, define 
\[
\alpha'\colon \Z^{n-1} \longrightarrow  \N
\]
by the formula 
\begin{multline*}
\alpha'(k_1,\dots, k_{n-1})
=
\alpha(k_1,\dots, k_{n-1},0)+ \alpha(k_1,\dots, k_{n-1},1)
\\
   +\cdots + \alpha(k_1,\dots, k_{n-1},k) .
\end{multline*}
Using the standard isomorphism
\[
A_{\alpha'(k_1,\dots, k_{n-1},0)}
\cong A_{\alpha(k_1,\dots, k_{n-1},0)}\otimes A_{\alpha(k_1,\dots, k_{n-1},1)}\otimes \cdots \otimes A_{\alpha(k_1,\dots, k_{n-1},k)}
\]
(both sides are tensor products of the same number of copies of $M_2(\C)$), we obtain an isomorphism
\[
A_\alpha \bigl (\Z^{n-1}{\times} \{ 0,\dots, k\}\bigr ) \cong A_{\alpha'} (\Z^{n-1}  ) 
\]
that induces a morphism (in fact an isomorphism)
\[
\tau\colon \lgd_\alpha \bigl(\Z^{n-1}{\times} \{ 0,\dots, k\}\bigr ) \longrightarrow \lgd_{\alpha'} (\Z^{n-1}  ) .
\]
The composition
\[
\xymatrix{
\lgd_{\alpha} (\Z^{n-1}  ) \ar[r]^-j & 
\lgd_\alpha \bigl(\Z^{n-1}{\times} \{ 0,\dots, k\}\bigr ) \ar[r]^-\tau  &
\lgd_{\alpha'} (\Z^{n-1}  )
}
\]
is equal to the standard inclusion 
\[
\iota_{\alpha',\alpha}\colon \lgd_{\alpha} (\Z^{n-1}  ) \longrightarrow 
\lgd_{\alpha'} (\Z^{n-1}  ),
\]
while by Theorem~\ref{thm-using-the-flip}, the composition 
\[
\xymatrix{
\lgd_\alpha \bigl(\Z^{n-1}{\times} \{ 0,\dots, k\}\bigr ) \ar[r]^-\tau  &
\lgd_{\alpha'} (\Z^{n-1}  ) \ar[r]^-j &
\lgd_{\alpha'} \bigl(\Z^{n-1}{\times} \{ 0,\dots, k\}\bigr ) 
}
\]
followed by the standard inclusion into $\lgd_{2\alpha'} \bigl(\Z^{n-1}{\times} \{ 0,\dots, k\}\bigr ) $ induces the same map in homology as the standard inclusion
\[
\iota_{2\alpha',\alpha}\colon \lgd_{\alpha} (\Z^{n-1}{\times} \{ 0,\dots, k\}  ) \longrightarrow 
\lgd_{2\alpha'} (\Z^{n-1} {\times} \{ 0,\dots, k\} ),
\]
Taking a direct limit over all $\alpha$, it follows that the morphism $j$ in \eqref{eq-the-morphism-j} induces an isomorphism in homology, as required.
\end{proof}

\section{Locally generated derivations of n-space}
\label{sec-homology-of-lgd-for-n-space}
In this section we shall   compute the  homology of the Lie algebra $\lgd(\Z^n)$  with trivial coefficients $\C$, for all $n\ge 0$. 

\subsection{Half-spaces (Eilenberg swindle)}
\label{sec-eilenberg-swindle}

The proof of the following theorem  is a variation on 
Eilenberg swindle arguments used in controlled topology; see for instance \cite{PedersenWeibel85} for perhaps the earliest example.   We need only adapt those arguments using  the abstract Eilenberg swindle method described  in Section~\ref{sec-abstract-eilenberg-swindle}.

\begin{theorem}
\label{thm-eilenberg-swindle-for-half-spaces}
If  $n\ge 1$, then
\[
H_p(\lgd(\Z^n_+),\C)  = 0  \qquad \forall p > 0.
\]
\end{theorem}

\begin{proof}
Let $n\ge 1$.  In view of Theorem~\ref{thm-absolute-v-relative} it suffices to prove that  
\[
H_p(\lgd(\Z^n,\Z^n_{+}), \C)  = 0 \qquad \forall p > 0 .
\]
For $\mathfrak{g} =  \lgd(\Z^n_+)$, we define a Lie algebra morphism
\[
 \mathrm{id}^\infty : \mathfrak{g}   \longrightarrow \mathfrak{g} 
\]
 by means of the formula 
\[
\mathrm{id}^\infty(\delta)  = 
 \gamma_1(s(\delta)) +  \gamma_2(s^2(\delta))+  \gamma_3(s^3(\delta))
 + \cdots ,
\]
where $\gamma_1,\gamma_2,\dots $ are 1-1 maps from $\N$ to itself with disjoint ranges. Note that this map is well-defined, in that the image really does lie in  $\lgd(\Z_+^n)$.

According to Lemma~\ref{lem-abstract-swindle}, to prove the theorem it suffices to show that 
\[
(\mathrm{id} \star \mathrm{id}^\infty)_* = \mathrm{id}^{\infty,*} \colon H_*(\mathfrak{g}, \C) \longrightarrow H_*(\mathfrak{g}, \C).
\]
But there is a correcting morphism $\gamma\colon \mathfrak{g} \to \mathfrak{g}$ such that the diagram 
\[
\xymatrix{
\mathfrak{g}\ar[r]^-{\Delta} & \mathfrak{g}\times \mathfrak{g}
\ar[r] ^{(\mathrm{id}\,,\,\mathrm{id}^\infty)} &\mathfrak{g}\times \mathfrak{g}\ar[r]^-{\mu} & 
\mathfrak{g} \ar[d]^\gamma
\\
\mathfrak{g}\ar[u]^s \ar[rrr]_{\mathrm{id}^\infty} &&   & \mathfrak{g} 
}
\]
commutes exactly, at the level of Lie algebras, so the theorem is proved by an appeal to Theorems~\ref{thm-gamma-morphisms-trivial-on-homology} and~\ref{thm-s-star-equals-1}. 
\end{proof}

\subsection{Suspension isomorphism for the primitive part of the homology of the Lie algebra of locally generated derivations}
We are now ready to present the main result of this paper:

\begin{theorem}
\label{thm-suspension-isomorphism}
  Let $n \ge 1$.   The primitive parts of the homology of $\lgd (\Z^n)$  and of $\lgd (\Z^{n-1})$ are related as follows: 
  \[
  \Prim (H_p (\lgd(\Z^{n}), \C)) \cong 
  \begin{cases}
       0 & p = 1
       \\
       \Prim(H_{p-1} (\lgd(\Z^{n-1}), \C)) & p > 1.
  \end{cases}
  \]
\end{theorem}

For the reader's convenience, the following proposition gathers results that we have proved up to this point, and that will be used in the proof of Theorem~\ref{thm-suspension-isomorphism}:

\begin{proposition} 
\label{prop-summary-result}
Let $n\ge 1$. 
\begin{enumerate}[\rm (i)]

\item The inclusion morphisms
 The inclusions 
    \[
    \lgd(\Z^{n-1}) \to \lgd(\Z^n_{+},\Z^{n-1})
    \quad\text{and} \quad 
    \lgd(\Z_+^n) \to \lgd(\Z^n,\Z_+^n)
    \]
    induce  isomorphisms on homology with coefficients in the trivial module $\C$.

\item 
The inclusion morphism
\[
\lgd(\Z^n_-)/ \lgd(\Z^n_-, \Z^{n-1}) \longrightarrow \lgd (\Z^n)/\lgd(\Z^n,\Z^n_+)
\]
is an isomorphism, and in particular it induces an isomorphism in homology with coefficients in the trivial module $\C$.

\end{enumerate}
\end{proposition} 

\begin{proof}
    Item (i) is Theorem~\ref{thm-absolute-v-relative}, and item (ii) is the excision isomorphism in Theorem~\ref{thm-excision-isomorphism}.
\end{proof} 

We shall need to add one more preparatory result to the above list, which is a consequence of the lemma below.

\begin{lemma}
\label{lem-spectral-sequence-comparison}
    Suppose given a commuting diagram of Lie algebras and Lie algebra morphisms,
    \[
    \xymatrix{
    \mathfrak{h}_0\ar[r] \ar[d] & \mathfrak{g}_0 \ar[d]
    \\
     \mathfrak{h}_1 \ar[r]   & \mathfrak{g}_1  
    }
    \]
     in which both of the  horizontal morphisms are inclusions of ideals. If the morphism $\mathfrak{h}_0 \to \mathfrak{h}_1$ and the associated morphism   $\mathfrak{g}_0/\mathfrak{h}_0 \to \mathfrak{g}_1/\mathfrak{h}_1$  both induce isomorphisms in homology with trivial coefficients $\C$, then so does the morphism $\mathfrak{g}_0\to \mathfrak{g}_1$.
\end{lemma}

\begin{proof} 
Denote by $\{ E^r_{*,*}(\mathfrak{g}_0,\mathfrak{h}_0)\}_{r\ge 0}$ and $\{ E^r_{*,*}(\mathfrak{g}_1,\mathfrak{h}_1)\}_{r\ge 0}$ the Hochschild-Serre spectral sequences for the pairs $(\mathfrak{g}_0,\mathfrak{h}_0)$ and $(\mathfrak{g}_1,\mathfrak{h}_1)$, as reviewed in Section \ref{sec-hochschild-serre}. The diagram in the statement of the theorem gives a morphism from the first spectral sequence to the second, and the hypotheses of the theorem imply that this morphism is an isomorphism of $E^2$-pages. So by the spectral sequence isomorphism theorem (see for instance \cite[Thm.~3.4]{McCleary}) the morphism is an isomorphism between $E^\infty$-terms. Hence the induced morphism 
\[
H_*(\mathfrak{g}_0, \C) \longrightarrow H_*(\mathfrak{g}_1, \C)
\]
is an isomorphism, too, since both sides have compatible filtrations (finite in each degree), for which the morphism is an isomorphism on subquotients.
\end{proof}

\begin{proposition} 
\label{prop-quotient-by-acyclic-ideal}
Let $n\ge 1$. The quotient morphism 
\[
\lgd (\Z^n) \longrightarrow \lgd(\Z^n) / \lgd (\Z^n, \Z^n_+)
\]
induces an isomorphism in homology with trivial coefficients $\C$.
\end{proposition}

\begin{proof} 
Apply Lemma~\ref{lem-spectral-sequence-comparison} to the commuting diagram 
\[
\xymatrix{
\lgd(\Z^n,\Z^n_+) \ar[r]\ar[d] & \lgd (\Z^n)\ar[d] 
\\
0\ar[r] &  \lgd (\Z^n)/\lgd(\Z^n,\Z^n_+) 
}
\]
using part (i)  of Proposition~\ref{prop-summary-result} and Theorem~\ref{thm-eilenberg-swindle-for-half-spaces} to verify that the left-hand vertical map induces an isomorphism in homology.
\end{proof}

\begin{proof}[Proof of Theorem~\ref{thm-suspension-isomorphism}] 
Thanks to Proposition~\ref{prop-quotient-by-acyclic-ideal}, the quotient morphism 
\[
\lgd(\Z^{n}) \to\lgd(\Z^{n})/\lgd(\Z^n,\Z^n_+)
\]
induces an isomorphism 
\[
  H_*  \bigl(\lgd(\Z^{n}), \C\bigr )  \stackrel \cong \longrightarrow 
   H_* \bigl(\lgd(\Z^{n})/\lgd(\Z^n,\Z^n_+) , \C\bigr ),
\]
while according to part (ii) of Proposition~\ref{prop-summary-result} the morphism 
\[
\lgd(\Z^n_-)/ \lgd(\Z^n_-, \Z^{n-1}) \longrightarrow \lgd (\Z^n)/\lgd(\Z^n,\Z^n_+)
\]
is an isomorphism, and so in particular it induces an isomorphism 
\begin{multline}
\label{eq-homology-isomorphism-from-excision}
H_*\bigl (\lgd(\Z^n_-)/ \lgd(\Z^n_-, \Z^{n-1}), \C\bigr ) \\
\stackrel \cong \longrightarrow 
H_*\bigl ( \lgd (\Z^n)/\lgd(\Z^n,\Z^n_+), \C\bigr ) .
\end{multline}
According to Theorems~\ref{thm-primitive-element-theorem} (the primitive element theorem)   and \ref{thm-eilenberg-swindle-for-half-spaces} (applied to $\Z^n_-$ rather than $\Z^n_+$), there is an isomorphism
\begin{multline*}
\Prim \Bigl (H_{p+1}\bigl (\lgd(\Z^n_-)/ \lgd(\Z^n_-, \Z^{n-1}), \C\bigr )\Bigr) 
\\
\cong 
\Prim \Bigl (H_p\bigl (  \lgd(\Z^n_-,\Z^{n-1}), \C\bigr )\Bigr) 
\end{multline*}
for all $p\ge 0$.
Combining this isomorphism with \eqref{eq-homology-isomorphism-from-excision}, and then using part (i) of Proposition~\ref{prop-summary-result} we obtain isomorphisms
\[
\begin{aligned}
\Prim \Bigl (H_{p+1}\bigl (\lgd(\Z^n), \C\bigr )\Bigr) 
    & \cong 
\Prim \Bigl (H_p\bigl (  \lgd(\Z^n_-,\Z^{n-1}), \C\bigr )\Bigr) 
\\
    & \cong \Prim \Bigl (H_p\bigl (  \lgd(\Z^{n-1}), \C\bigr )\Bigr) ,
\end{aligned}
\]
for all $p\ge 0$, as required.
\end{proof} 

\begin{theorem}
\label{thm-explicit-formula-for-prinitive-part}
If $n\ge 0$, then
    \[
    \Prim \Bigl (H_{p}\bigl (\lgd(\Z^n), \C\bigr )\Bigr)  =
    \begin{cases}
        \C & p= n{+}3, n{+}5,n{+}7\dots 
        \\
         0 & \text{otherwise.}
    \end{cases}
    \]
\end{theorem}

\bibliographystyle{alpha}
\bibliography{refs}

\end{document}